\documentclass[a4paper]{amsart}
\usepackage{color}
\usepackage{amsmath}
\usepackage{amssymb}
\usepackage[all]{xy}
\usepackage{mathrsfs}
\usepackage{enumerate}
\usepackage{nicefrac}
\usepackage{graphicx}
\usepackage[dvips]{geometry}

\newcommand{\D}{\mathcal{D}}
\newcommand{\bangle}[1]{\langle #1 \rangle}

\newtheorem{theorem}{Theorem}[section]
\newtheorem{lemma}[theorem]{Lemma}
\newtheorem{proposition}[theorem]{Proposition}
\newtheorem{corollary}[theorem]{Corollary}
{\theoremstyle{remark}
\newtheorem{remark}[theorem]{Remark}

}

\theoremstyle{definition}
\newtheorem{construction}[theorem]{Construction}
\newtheorem{definition}[theorem]{Definition}
\newtheorem{example}[theorem]{Example}

\newcommand{\CC}{\mathbf{C}}

\newcommand{\QQ}{\mathbf{Q}}
\newcommand{\ZZ}{\mathbf{Z}}
\newcommand{\NN}{\mathbf{N}}
\newcommand{\PP}{\mathbf{P}}
\newcommand{\A}{\mathbf{A}}
\newcommand{\m}{\mathfrak{m}}
\newcommand{\lin}[1]{\text{span}\,(#1)}
\newcommand{\equivn}{{\stackrel{\scriptscriptstyle\text{num}}{\sim}}}
\newcommand{\KK}{k}

\DeclareMathOperator{\codim}{codim}

\DeclareMathOperator{\face}{face}
\DeclareMathOperator{\quot}{Quot}
\DeclareMathOperator{\rk}{rank}
\DeclareMathOperator{\relint}{relint}
\DeclareMathOperator{\tail}{tail}

\DeclareMathOperator{\Hom}{Hom}
\DeclareMathOperator{\divisor}{div}

\DeclareMathOperator{\wdiv}{Div}

\DeclareMathOperator{\spec}{Spec}
\DeclareMathOperator{\proj}{Proj}
\DeclareMathOperator{\ord}{ord}
\DeclareMathOperator{\supp}{supp}

\DeclareMathOperator{\SPEC}{\mathbf{Spec}}
\DeclareMathOperator{\conv}{conv}
\DeclareMathOperator{\pos}{pos}

\DeclareMathOperator{\cl}{Cl}
\DeclareMathOperator{\num}{Num}

\DeclareMathOperator{\numq}{Num_{\hspace{.2pt}\QQ}}

\newcommand{\CO}{{\mathcal{O}}}

\newcommand{\OO}{{\mathcal{O}}}
\newcommand{\DD}{\mathcal{D}}
\newcommand{\TT}[0]{T}
\newcommand{\AF}{\mathbf{A}}
\newcommand{\tA}[0]{\ensuremath{\widetilde{A}}}
\newcommand{\tX}[0]{{\ensuremath{\widetilde{X}}}}
\newcommand{\tY}[0]{{\ensuremath{\widetilde{Y}}}}
\newcommand{\divi}[0]{\ensuremath{\operatorname{div}}}
\newcommand{\FF}[0]{\ensuremath{\mathscr{F}}}
\newcommand{\pol}[0]{\ensuremath{\operatorname{Pol}}}
\newcommand{\rank}[0]{\ensuremath{\operatorname{rank}}}

\begin{document}
\title[Normal singularities with torus actions]{Normal singularities
  with torus actions}

\author[A.~Liendo]{Alvaro Liendo}
\address{Universit\'e Grenoble I, 
	Institut Fourier, 
	UMR 5582 CNRS-UJF, 
	BP 74, 
	38402 St.\ Martin d'H\`eres c\'edex, 
	France}
\email{alvaro.liendo@ujf-grenoble.fr}

\author[H.~S\"uss]{Hendrik S\"u{\ss}}
\address{Institut f\"ur Mathematik,
        LS Algebra und Geometrie,
        Brandenburgische Technische Universit\"at Cottbus,
        PF 10 13 44, 
        03013 Cottbus, Germany}
\email{suess@math.tu-cottbus.de}

\thanks{
  \mbox{\hspace{11pt}}{\it Key words}: torus actions, T-varieties,
  characterization of singularities, toroidal desingularization.}

\begin{abstract}
  We propose a method to compute a desingularization of a normal
  affine variety $X$ endowed with a torus action in terms of a
  combinatorial description of such a variety due to Altmann and
  Hausen. This desingularization allows us to study the structure of
  the singularities of $X$. In particular, we give criteria for $X$ to
  have only rational, ($\QQ$-)factorial, or ($\QQ$-)Gorenstein
  singularities. We also give partial criteria for $X$ to be
  Cohen-Macaulay or log-terminal.

  Finally, we provide a method to construct factorial affine varieties
  with a torus action. This leads to a full classification of such
  varieties in the case where the action is of complexity one.
\end{abstract}

\maketitle
\tableofcontents

\section*{Introduction}

The theory of singularities on toric varieties is well
established. All toric singularities are log-terminal and thus rational
and Cohen-Macaulay. Furthermore, there are explicit combinatorial
criteria to decide if a given toric variety is ($\QQ$-)factorial or
($\QQ$-)Gorenstein, see \cite{Dai02}. In this paper we elaborate the
analog criteria for more general varieties admitting torus actions.

Let $X$ be a normal variety endowed with an effective torus
action. The complexity of this action is the codimension of the
maximal orbits. By a classic theorem of Sumihiro \cite{Sum74} every
point $x\in X$ posses an affine open neighborhood invariant under the
torus action. Hence, local problems can be reduced to the affine case.

There are well known combinatorial descriptions of normal
T-varieties. We refer the reader to \cite{Dem70} and
\cite{Oda88} for the case of toric varieties, to \cite[Ch. 2
and 4]{KKMS73} and \cite{Tim08} for the complexity one case, and to
\cite{MR2207875,AHS08} for the general case.

Let us fix some notation. We let $\KK$ be an algebraically closed
field of characteristic 0, $M$ be a lattice of rank $n$, and $\TT$ be
the algebraic torus $\TT=\spec\KK[M]\simeq(\KK^*)^n$. A T-variety
$X$ is a variety endowed with an effective algebraic action of
$\TT$. For an affine variety $X=\spec A$, introducing a $\TT$-action
on $X$ is the same as endowing $A$ with an $M$-grading. 

We let $N_{\QQ}=N\otimes\QQ$, where $N=\Hom(M,\ZZ)$ is the dual
lattice of $M$. Any affine toric variety can be described via a
polyhedral cone $\sigma\subseteq N_{\QQ}$. Similarly, the
combinatorial description of normal affine T-varieties due to Altmann
and Hausen \cite{MR2207875} involves the data $(Y,\sigma,\DD)$ where
$Y$ is a normal semiprojective variety, $\sigma\subseteq N_{\QQ}$ is a
polyhedral cone, and $\DD$ is a polyhedral divisor on $Y$ i.e., a
divisor whose coefficients are polyhedra in $N_{\QQ}$ with tail cone
$\sigma$.

The normal affine variety corresponding to the data $(Y,\sigma,\DD)$
is denoted by $X[\DD]$. The construction involves another normal
variety $\tX[\DD]$, which is affine over $Y$, and a proper birational
morphism $r:\tX[\DD]\rightarrow X[\DD]$, see Section
\ref{preliminaries} for more details.

This description is not unique. In Section \ref{sec:can-desing}, we
show that for every T-variety $X$ there exists a polyhedral divisor
$\DD$ such that $X=X[\DD]$ and $\tX[\DD]$ is a toroidal
variety. Hence, the morphism $r:\tX[\DD]\rightarrow X[\DD]$ is a
partial desingularization of $X$ having only toric
singularities. Moreover, a desingularization of $\tX[\DD]$ can be
obtained via toric methods and thus also a desingularization of
$X[\DD]$, see Construction~\ref{sec:constr-desing}.

Let $X$ be a normal variety and let $\psi:Z\rightarrow X$ be a
desingularization. Usually, the classification of singularities
involves the higher direct images of the structure sheaf
$R^i\psi_*\OO_Z$. In particular, $X$ has rational singularities if
$R^i\psi_*\OO_Z=0$ for all $i\geq 1$, see e.g., \cite{Art66,Elk78}. In
section \ref{sec:rational-sing} we compute the higher direct image
sheaves $R^i\psi_*\OO_Z$ for a T-variety $X[\DD]$ in terms of
the combinatorial data and we give a criterion for $X[\DD]$ to have
rational singularities.

A well known theorem of Kempf \cite[p. 50]{KKMS73} states that a
variety $X$ has rational singularities if and only if $X$ is
Cohen-Macaulay and the induced map $\psi_*:\omega_Z\hookrightarrow
\omega_X$ is an isomorphism. In Proposition \ref{cmgen} we apply
Kempf's Theorem to give a partial characterization of T-varieties
having Cohen-Macaulay singularities.

Invariant $\TT$-divisors were studied in \cite{PeSu08}. In particular,
a description of the class group, and a representative of the
canonical class of $X[\DD]$ are given. In Section
\ref{sec:canonical-div} we use this results to state necessary and
sufficient conditions for $X[\DD]$ to be ($\QQ$-)factorial or
($\QQ$-)Gorenstein in terms of the combinatorial data. Furthermore, in
Theorem \ref{sec:thm-log-terminal} we apply the partial
desingularization obtained in Section \ref{sec:can-desing} to give a
criterion for $X[\DD]$ to have log-terminal singularities.

In \cite{Wat81} some of the results in Sections
\ref{sec:rational-sing} and \ref{sec:canonical-div} were proved for a
1-dimensional torus action on $X$. Our results can be seen as the
natural generalization of these results of Watanabe, see also
\cite[Section 4]{FlZa03c}.

In Section \ref{sec:complexity-one} we specialize our results in
Sections \ref{sec:rational-sing} and \ref{sec:canonical-div} for a
T-variety $X[\DD]$ of complexity one. In this case, the variety $Y$ in
the combinatorial data is a smooth curve. This make the criteria more
explicit. In particular, if $X[\DD]$ has $\QQ$-Gorenstein or rational
singularities, then $Y$ is either affine or the projective line.

Finally, in Section \ref{sec:fact-t-vari} we provide a method to
construct factorial T-varieties based on the criterion for
factoriality given in Proposition \ref{sec:prop-factorial}. In the
case of complexity one, this method leads to a full classification of
factorial quasihomogeneous T-varieties analogous to the one given in
\cite{HHS09}. A common way to show that an affine variety is factorial
is to apply the criterion of Samuel \cite{MR0214579} or the
generalization by Scheja and Storch \cite{MR751470}.  However, for the
majority of the factorial varieties that we construct with our method
these criteria do not work.

In the entire paper the term variety means a normal integral scheme of
finite type over an algebraically closed field $\KK$ of characteristic
0.

\vspace{1em} \textbf{Acknowledgements.} We would like to thank Yuri
Prokhorov for kindly and patiently answering our questions and
therefore helping us to finish this article and to overcome some
inexcusable lacks in our knowledge on birational geometry.  We also
thank Nathan Ilten. His suggestions helped us to improve this paper.

\section{Preliminaries} \label{preliminaries}
First, we fix some notation. In this paper $N$ is always a lattice of rank $n$, and $M = \Hom(N,\ZZ)$ is its dual. The associated rational vector spaces are denoted by $N_\QQ := N \otimes \QQ$ and $M_\QQ := M \otimes \QQ$. Moreover, $\sigma \subseteq N_\QQ$ is a pointed convex polyhedral cone, and $\sigma^\vee \subseteq M_\QQ$ is its dual cone. Let  $\sigma^\vee_M:=\sigma^\vee \cap M$ be the semigroup of lattice points inside $\sigma^\vee$.

We consider convex polyhedra $\Delta \subseteq N_\QQ$ admitting a decomposition as Minkowski sum $\Delta = \Pi + \sigma$ with a compact polyhedron $\Pi \subseteq N_\QQ$; we refer to $\sigma$ as the {\em tail cone\/} of $\Delta$ and refer to $\Delta$ as a {\em $\sigma$-polyhedron}. We denote the set of all $\sigma$-polyhedra by $\pol_\sigma(N_{\QQ})$. With respect to Minkowski addition, $\pol_\sigma(N_{\QQ})$ is a semigroup with neutral element $\sigma$.

We are now going to describe affine varieties with an action of the torus 
$T = \spec \KK[M]$.
Let $Y$ be a normal variety, which is semiprojective, i.e. projective over an affine variety. Fix a pointed convex  polyhedral cone $\sigma \subseteq N_\QQ$.
A {\em polyhedral divisor\/} on $Y$
is a formal finite sum
\begin{eqnarray*}
\D 
& = & 
\sum_Z \Delta_Z \cdot Z,
\end{eqnarray*}
where $Z$ runs over the prime divisors of $Y$ and the coefficients $\Delta_Z$ are all $\sigma$-polyhedra with $\Delta_Z=\sigma$ for all but finitely many of them.

For every $u \in \sigma^\vee_M$ we have the 
evaluation
\begin{eqnarray*}
\D(u)
& := & 
\sum_Z \min_{v \in \Delta_Z} \bangle{u ,v} \cdot Z,
\end{eqnarray*}
which is a $\QQ$-divisor living on $Y$.  This defines an evaluation
map $\D^\vee :\sigma^\vee \rightarrow \wdiv_\QQ Y$, which is piecewise
linear and the loci of linearity are (not necessarily pointed)
subcones of $\sigma^\vee$. Hence, $\D^\vee$ defines a quasifan which
subdivides $\sigma^\vee$. We call it the normal quasifan of $\D$.

We call the polyhedral
divisor $\D$ on $Y$ {\em proper\/} if the following conditions hold:
\begin{enumerate}
\item The divisor $\D(u)$ has a base point free multiple for $u \in
  \sigma^\vee_M$.
\item The divisor $\D(u)$ is big for $u\in \relint \sigma^\vee\cap M$.
\end{enumerate}

By construction, every polyhedral divisor $\D$ on a normal variety $Y$
defines a sheaf $\mathcal{A}[\D]$ of $M$-graded
$\mathcal{O}_{Y}$-algebras and its ring $A[\D]$ of global sections:
$$ 
\mathcal{A}[\D]
\ := \ 
\bigoplus_{u \in \sigma_M^\vee} \mathcal{O}(\D(u))\cdot\chi^u,
\qquad\qquad
A[\D]
\ := \
H^0(Y, \mathcal{A}[\D]).
$$

Now suppose that $\D$ is proper. Theorem 3.1 in \cite{MR2207875}
guarantee that $A[\D]$ is a normal affine algebra. Thus, we obtain an
affine varieties $X:=X[\D] := \spec A[\D]$ and
$\widetilde{X}:=\widetilde{X}[\D]:=\SPEC_{Y} \mathcal{A}[\D]$. Both
varieties $X$ and $\tX$ come with an effective action of the torus $T
= \spec \KK[M]$ and there is a proper birational morphism
$r:\widetilde{X} \rightarrow X$. Moreover, by definition of
$\widetilde{X}$ there is an affine morphism
$q:\widetilde{X}\rightarrow Y$ and the composition
\[\pi:=q\circ r^{-1}:X \dashrightarrow Y\]
is a rational map defined outside a closed subset of codimension at
least $2$.

Note that there is a natural inclusion $A[\D] \subset \bigoplus_{u \in
  M} K(Y) \cdot \chi^u$ which gives rise to a standard representation
$f \cdot \chi^u$ with $f \in K(Y)$ and $u \in M$ for every
semi-invariant rational function from $K(X) = K(\widetilde{X})$. With
this notation the rational map $\pi$ is given by the natural inclusion
of function field
\[\textstyle K(Y) \subset K(X) = \quot\left(\bigoplus_u K(Y) \cdot \chi^u\right).\]

By Theorem 3.4 in \cite{MR2207875}, every normal affine variety with
an effective torus action arises from a proper polyhedral divisor.

\begin{example}
  Letting $N=\ZZ^2$ and $\sigma=\pos((1,0),(1,6))$, in $N_{\QQ}=\QQ^2$
  we consider the $\sigma$-polyhedra $\Delta_0=\conv((1,0),
  (1,1))+\sigma$, $\Delta_1=(-\nicefrac{1}{2},0)+\sigma$, and
  $\Delta_1=(-\nicefrac{1}{3},0)+\sigma$.
  \begin{figure}[!ht]
    \centering
    \includegraphics[scale=0.8]{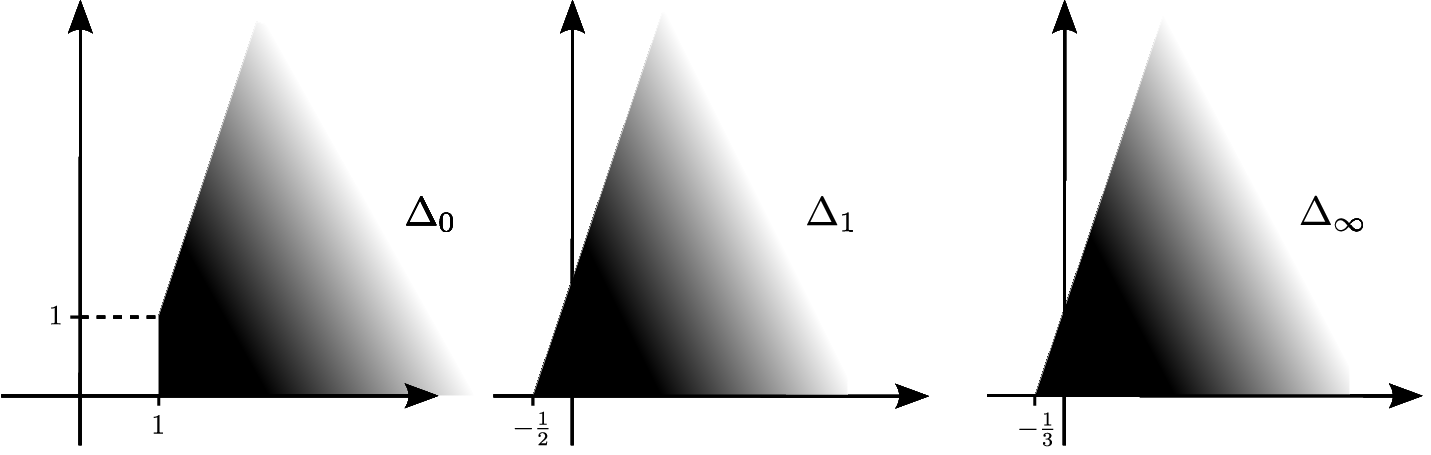} 
  \end{figure}

  Let $Y=\PP^1$ so that $K(Y)=\KK(t)$, where $t$ is a local coordinate
  at zero. We consider the polyhedral divisor
  $\DD=\Delta_0\cdot[0]+\Delta_1\cdot[1]+\Delta_\infty\cdot[\infty]$,
  and we let $A=A[\DD]$ and $X=\spec\,A$. An easy calculation shows
  that the elements
  $$u_1=\chi^{(0,1)},\quad u_2=\frac{t-1}{t^2}\chi^{(2,0)},\quad
  u_3=\frac{(t-1)^2}{t^3}\chi^{(3,0)}, \quad \mbox{and} \quad
  u_4=\frac{(t-1)^3}{t^5}\chi^{(6,-1)}$$ generate $A$ as an algebra. Furthermore,
  they satisfy the irreducible relation $u_2^3-u_3^2+u_1u_4=0$, and
  so
     \begin{align*} \label{iso-ex1}
    A\simeq\KK[x_1,x_2,x_3,x_4]/(x_2^3-x_3^2+x_1x_4)\,.
  \end{align*}
\end{example}

For a polyhedral divisor $\D$ and a (not necessarily closed) point $y
\in Y$ we define the {\em slice} of $\D$ at $y$ by $\D_y:=\sum_{Z
  \supset y} \Delta_Z$. Note, that via $\D_Z$ we may recover the
polyhedral coefficients of $\D$.

We want to describe the exceptional divisor of the morphism
$\widetilde{X}[\D] \rightarrow X[\D]$. In general on a T-variety the
two types of prime divisors.  Prime divisors of {\em horizontal} type
consist of orbit closures of dimension $\dim N -1$ and prime divisors
of {\em vertical} type of orbit closures of dimension $\dim N$. Note,
that a generic point on a vertical prime divisor has a finite isotropy
group, while on a horizontal prime divisor every point has infinite
isotropy.

Let $\rho \in \sigma(1)$ be a ray of the tail cone. We call it an {\em
  extremal ray} of $\D$ if $\D(u)$ is big for $u \in
\relint(\sigma^\vee \cap \rho^\perp)$. The set of extremal rays is
denoted by $\D^\times$. For a vertex $v \in \D_Z^{(0)}$ we consider
the smallest natural number $\mu(v)$ such that $\mu(v)\cdot \mu$ is a
lattice point.
A vertex $v$ is called {\em extremal} if $\D(u)|_Z$ is big for every
$u$ from the interior of the normal cone
\[\mathcal{N}(\Delta_Z,v) = \left\{u \mid \forall_{w \in \Delta}: \langle u, w -v \rangle > 0 \right\}.\]
 The set of extremal
vertices in $\D_Z$ is denoted by $\D_Z^\times$. 

\begin{theorem}[Prop.~3.13, \cite{PeSu08}] \label{th:divisor}
 For the invariant prime divisors on $\tX[\D]$ there are bijections
  \begin{enumerate}
\item between rays $\rho$ in $\sigma(1)$ and vertical prime divisors $\widetilde{E}_\rho$ of $\widetilde{X}[\D]$.
\item between pairs $(Z,v)$, where $Z$ is a prime divisor on $Y$ and
  $v$ is a vertex in $\D_Z$, and horizontal prime divisors
  $\widetilde{D}_{Z,v}$ of $\widetilde{X}[\D]$. 
  \end{enumerate}
  Via this correspondences the non-exceptional invariant divisor of
  $\tX[\D] \rightarrow X[\D]$, and therefore the invariant divisors
  $D_\rho$, $D_{Z,v}$ on $X[\D]$ correspond to the elements of $\rho
  \in \D^\times$ or $v \in \D^\times_Z$, respectively.
\end{theorem}

For a semi-invariant function $f \cdot \chi^u$ the corresponding
invariant principal divisor on $X[\D]$ is
\begin{equation}
  \label{eq:principal-divisor}
  \sum_{Z,v} \mu(v)(\bangle{u,v} + \ord f) \cdot D_{Z,v} + \sum_\rho \bangle{u,n_\rho} \cdot E_\rho.
\end{equation}
Hence, for the pullback of a prime divisor $Z$ on $Y$ to $\tX[\D]$ or $X[\D]$, respectively,  we obtain
\[q^*Z = \sum_{v \in \D_Z^{(0)}} \mu(v) \cdot \widetilde{D}_{Z,v}, \quad \pi^*Z = \sum_{v \in \D_Z^\times} \mu(v) \cdot D_{Z,v}.\]

\section{Toroidal desingularization}
\label{sec:can-desing}

The combinatorial description of affine T-varieties in Section \ref{preliminaries} is not unique. The following Lemma is a specialization of Corollary 8.12 in \cite{MR2207875}. For the convenience of the reader, we provide a short argument.

\begin{lemma}\label{proj}
  Let $\DD$ be a proper polyhedral divisor on a normal variety
  $Y$. Then for any projective birational morphism
  $\psi:\tY\rightarrow Y$ the variety $X[\DD]$ is equivariantly
  isomorphic to $X[\psi^*\DD]$.
\end{lemma}

\begin{proof} 
  We only need to show that
$$H^0(Y,\OO_Y(\DD(u)))\simeq H^0(\tY,\OO_\tY(\psi^*\DD(u))),\mbox{ for all }u\in\sigma^\vee_M\,.$$
We let $r$ be such that $r\DD(u)$ is Cartier $\forall
u\in\sigma_M^\vee$.
By Zariski's main theorem $\psi_*\OO_\tY=\OO_Y$ and by the projection
formula, for all $u\in\sigma_M^\vee$ we have
\begin{align*}
  H^0(Y,\OO_Y(\DD(u)))&\simeq\big\{f\in \KK(\tY): f^r\in
  H^0(\tY,\OO_\tY(\psi^*r\DD(u)))\big\}=H^0(\tY,\OO_\tY(\psi^*\DD(u)))\,.
\end{align*}
\end{proof}

In the previous Lemma, $\tX=\tX[\DD]$ is not equivariantly isomorphic
to $\tX[\psi^*\DD]$, unless $\psi$ is an isomorphism.

\begin{definition} \label{support}
We define the \emph{support} of a proper polyhedral divisor as 
$$\supp\DD= \left\{Z\mbox{ prime divisor} \mid \D_Z\neq \sigma \mbox{ or }\D_Z^\times \neq \D^{(0)}_Z\right\}\,.$$
We say that $\DD$ is an
\begin{enumerate}
\item  \emph{SNC polyhedral divisor} if $\DD$ is proper, $Y$ is smooth, and $\supp\DD$ is a simple normal crossing (SNC) divisor,
\item \emph{completely ample} if $\D(u)$ is ample for \emph{every} $u
  \in \relint\sigma^\vee$. 
\end{enumerate}
\end{definition}

\begin{remark}
  The above notion of complete ampleness has the following geometric interpretations. 
$\D$ is completely ample iff $Y$ equals the Chow quotient $X[\D]/\!/T$
and $X[\D]$ has a trivial GIT chamber decomposition, i.e. there is exactly one GIT quotient of expected dimension.
Hence, the existence of a completely ample polyhedral
 divisor is a quite restrictive condition for a T-variety.
\end{remark}

In the case of complexity one i.e., when $Y$ is a curve, any proper
polyhedral divisor is SNC and completely ample.


\begin{corollary}
For any  T-variety $X$ there exists an SNC polyhedral divisor on a smooth variety $Y$ such that $X=X[\DD]$.
\end{corollary}
\begin{proof}
Let $\DD'$ be proper polyhedral divisor on a normal variety $Y'$ such that $X=\spec X[\DD']$. Let $\psi:Y\rightarrow Y'$ be a resolution of singularities of $Y$ such that $\supp\psi^*\DD'$ is SNC. By Chow Lemma we can assume that $Y$ is semiprojective. By Lemma \ref{proj}, $\DD=\psi^*\DD'$ is an SNC polyhedral divisor such that $X=X[\DD]$.
\end{proof}

Now we elaborate a method to effectively compute an equivariant partial desingularization of an affine T-variety in terms of the combinatorial data $(Y,\DD)$. A key ingredient for our results is the following example (Cf. Example 3.19 in \cite{Lie08}).

\begin{example} \label{extor}
  Let $H_i$, $i\in\{1,\ldots,n\}$ be the coordinate hyperplanes in
  $Y=\AF^n$, and let $\DD$ be the SNC divisor on $Y$ given by
  $$\DD=\sum_{i=0}^{n}\Delta_i\cdot H_i,\quad \mbox{where } \Delta_i\in\pol_\sigma(N_{\QQ})\,.$$
  For the function field of $Y$ we have $K(Y)=\KK(t_1,\ldots, t_n)$ and we obtain
  \begin{align*}
    H^0(Y,\OO_Y(\DD(u)))&=\big\{f\in \KK(Y)\mid\divi(f)+\DD(u)\geq 0\big\} \\
    &=\left\{f\in \KK(Y)\mid\divi(f)+\sum_{i=1}^n \min_{v \in \Delta_i} \langle u,
      v \rangle \cdot H_i\geq 0\right\} \\
    &=\bigoplus_{r_i\geq -h_i(u)}\KK\cdot t_1^{r_1}\cdots t_n^{r_n}\,.
  \end{align*}
  Let $N'=N \times \ZZ^n$, $M'=M\times \ZZ^n$ and $\sigma'$ be the
  Cayley cone in $N'_{\QQ}$, i.e. the cone spanned by $(\sigma,\overline{0})$ and
  $(\Delta_i,e_i)$, $\forall i\in\{1,\ldots,n\}$, where $e_i$ is the
  $i$-th vector in the standard base of $\QQ^n$. A vector $(u,r)\in
  M'$ belongs to the dual cone $(\sigma')^\vee$ if and only if
  $u\in\sigma^\vee$ and $r_i\geq -h_i(u)$.

With these definitions we have
$$A[\DD]=\bigoplus_{u\in\sigma_M^\vee} H^0(Y,\OO_Y(\DD(u)))=\bigoplus_{(u,r) \in(\sigma')^\vee\cap M'}k\cdot t_1^{r_1}\cdots t_n^{r_n}\simeq k[(\sigma')^\vee\cap M']\,.$$

Hence $X[\DD]$ is isomorphic as an abstract variety to the toric variety with cone $\sigma'\subseteq N'_{\QQ}$. Since $Y$ is affine $\tX\simeq X$ is also a toric variety.
\end{example}

Recall \cite{KKMS73} that a variety $X$ is toroidal if for every $x\in X$ there is a formal neighborhood isomorphic to a formal neighborhood of a point in an affine toric variety.

\begin{proposition} \label{toroidal}
  Let $\DD=\sum_Z \Delta_Z\cdot Z$ be a proper polyhedral divisor on a
  semiprojective normal variety $Y$. If $\DD$ is SNC then
  $\tX=\tX[\DD]$ is a toroidal variety.
\end{proposition}
\begin{proof}
  For $y\in Y$ we consider the fiber $X_y$ over $y$ for the morphism
  $\varphi:\tX\rightarrow Y$. We let also $\mathfrak{U}_y$ be a formal
  neighborhood of $X_y$.

  We let $n=\dim Y$ and $$S_y=\{Z \mbox{ prime divisor } \mid y\in
  Z\mbox{ and } \Delta_Z\neq\sigma\}\,.$$ Since $\supp\DD$ is SNC, we
  have that $\operatorname{card}(S_y)\leq n$. Letting
  $j:S_y\rightarrow \{1,\ldots,n\}$ be any injective function, we
  consider the smooth $\sigma$-polyhedral divisor
  $$\DD_y'= \sum_{Z \in S_y} \Delta_Z\cdot H_{j(Z)}, \quad \mbox{on} \quad \AF^n\,.$$

  Since $Y$ is smooth, $\mathfrak{U}_y$ is isomorphic to a formal
  neighborhood of the fiber over zero for the canonical morphism
  $\pi':\tX[\DD_y']=\SPEC_{\AF^n} \tA[\DD_y']\rightarrow
  \AF^n$. Finally, Example \ref{extor} shows that $\tX[\DD_y']$ is
  toric for all $y$ and so $X$ is toroidal. This completes the proof.
\end{proof}

\begin{remark}
Proposition \ref{toroidal} holds in the less restrictive case where only 
$$\{Z \mbox{ prime divisor}\mid \DD_Z\neq\sigma\}$$
is SNC. The definition of $\supp\DD$ given in Definition \ref{support}
will be useful in Section \ref{sec:canonical-div}.
\end{remark}

Since the morphism $\varphi:\tX[\DD]\rightarrow X[\DD]$ is proper and
birational, to obtain a desingularization of $X$ it is enough to have
a desingularization of $\tX$. If further $\DD$ is SNC, then $\tX$ is
toroidal and there exists a toric desingularization. Indeed, in this
case we can give a description in terms of polyhedral divisors of the
desingularization of $\tX$.

\begin{construction}
  \label{sec:constr-desing}
  Let $\D$ be an SNC polyhedral divisor on $Y$ and $\{U_i\}_{i\in I}$
  an open affine covering of $Y$ such that the components of $U_i \cap
  \supp \D$ meet in (at least) a common point $y_i$ for every $i \in
  I$. It is not hard to see, that the polyhedral divisors
  $\D_i:=\D|_{U_i}$ give an open affine covering of $\{X[\D_i]\}_i$ of
  $\tX[\D]$. Moreover, the singularities of $X[\D_i]$ are toric and
  correspond to the Cayley cone from example~\ref{extor}. Let $S_i$ be
  a set of simply normal crossing prime divisors having intersection $y_i \in
  U_i$ and containing the prime divisors of $U_i \cap \supp \D$. We
  set $N'=N \times \ZZ^{S_i}$ coming with canonical basis elements
  $e_Z$.

  The Cayley cone $\sigma_i$ is now the cone generated by $\Delta_Z
  \times \{e_Z\}$ for $Z \in S_i$ and $\sigma \times \{0\}$. The
  intersections of $\sigma_i$ with the linear subspaces $N_Z:=N_\QQ
  \times \langle e_Z \rangle$ results in cones isomorphic to
\[\sigma_Z := \conv\left(\sigma \times \{0\} + \Delta_Z \times \{1\}\right) \subset
N_\QQ \times \QQ.\] and the intersection with the affine subspace $N
\times \{e_Z\}$ results in an polyhedron isomorphic to $\Delta_Z$. We
will use these isomorphisms below.

We fix a toric projective desingularization $\Sigma_Z$ of $\sigma_Z$,
such that the induced desingularization of $\sigma=\sigma_Z \times
N_\QQ$ are the same. Now for every Cayley cone $\sigma_i$ there exists
a desingularization $\Sigma_i$ which is spanned by the
desingularizations of the facets $\{\sigma_Z\}_{Z \in S_i}$. For a
cone $\tau \in \Sigma_i$ we define the polyhedral divisor $\D_i^\tau =
\sum_{Z \in S_i} \left(\tau \cap \Delta_z\right) \cdot Z$. An open
affine covering of a desingularization of $X[\D]$ is now given by
\[\left\{ X[\D^\tau_i] \mid i\in I,\; \tau \in \Sigma_i \right\}.\]

Note, that the polyhedral divisors $\D^\tau_i$ form a so-called
divisorial fan, which is a non-affine generalization of a polyhedral
divisor similar to the passage from cones to fans in toric
geometry. For details see~\cite{AHS08}.
\end{construction}

\section{Higher direct images sheaves}
\label{sec:rational-sing}

In this section we apply the partial desingularization
$\varphi:\tX[\DD]\rightarrow X[\DD]$ to compute the higher direct
images of the structure sheaf of any desingularization $W$ of
$X[\DD]$. This allows us to provide information about the
singularities of $X$ in terms of the combinatorial data $(Y,\DD)$. We
recall the following notion.

\begin{definition}
  A variety $X$ has rational singularities if there exists a
  desingularization $\psi:W\rightarrow X$, such that
$$\psi_*\OO_W=\OO_X,\quad\mbox{and}\quad R^i\psi_*\OO_W=0, \quad \forall i>0\,.$$
\end{definition}

The sheaves $R^i\psi_*\OO_W$ are independent of the particular choice of a desingularization of $X$. The first condition $\psi_*\OO_W=\OO_X$ is equivalent to $X$ being normal.

The following well known lemma follows by applying the Leray spectral
sequence. For the convenience of the reader we provide a short
argument.
\begin{lemma} \label{rati}
Let $\varphi:\tX\rightarrow X$ be a proper surjective, birational morphism, and let $\psi:W\rightarrow X$ be a desingularization of $X$. If $\tX$ has only rational singularities, then
$$R^i\psi_*\OO_W=R^i\varphi_*\OO_{\tX}, \quad \forall i\geq0\,.$$
\end{lemma}
\begin{proof}
We may assume that the desingularization $\psi$ is such that $\psi=\varphi\circ\widetilde{\psi}$, where $\widetilde{\psi}:W\rightarrow\tX$ is a desingularization of $\tX$. The question is local on $X$, so we may assume that $X$ is affine. Then, by \cite[Ch. III, Prop. 8.5]{Har77} we have\footnote{As usual for a $A$-module $M$, $M^{\sim}$ denotes the associated sheaf on $X=\spec A$.}
$$R^i\psi_*\OO_W=H^i(W,\OO_W)^{\sim}\quad\mbox{and}\quad R^i\varphi_*\OO_{\tX}=H^i(\tX,\OO_{\tX})^{\sim},\quad \forall i\geq 0\,.$$
Since $\tX$ has rational singularities
$$\widetilde{\psi}_*\OO_W=\OO_{\tX},\quad\mbox{and}\quad R^i\widetilde{\psi}_*\OO_W=0, \quad \forall i>0\,.$$
By Leray spectral sequence for $(p,q)=(i,0)$ we have
$$H^i(W,\OO_W)=H^i(\tX,\widetilde{\psi}_*\OO_W)=H^i(\tX,\OO_{\tX}),\quad \forall i\geq 0\,,$$
proving the Lemma.
\end{proof}

In the following theorem for a T-variety $X=X[\DD]$ and a desingularization $\psi:W\rightarrow X$ we provide an expression for $R^i\psi_*\OO_Z$ in terms of the combinatorial data $(Y,\DD)$.

\begin{theorem} \label{Tdirect}
Let $X=X[\DD]$, where $\DD$ is an SNC polyhedral divisor on $Y$. If $\psi:W\rightarrow X$ is a desingularization, then for every $i\geq0$, the higher direct image $R^i\psi_*\OO_W$ is the sheaf associated to
$$\bigoplus_{u\in\sigma_M^\vee} H^i(Y,\OO(\DD(u)))$$
\end{theorem}

\begin{proof}
  Let $\psi:W\rightarrow X$ be a desingularization of $X$. Consider
  the proper birational morphism $\varphi:\tX:=\tX[\DD]\rightarrow
  X$. By Lemma \ref{toroidal} $\tX$ is toroidal, thus it has only
  toric singularities which are rational \cite{Dai02}. By Lemma
  \ref{rati} we have
  $$R^i\psi_*\OO_W=R^i\varphi_*\OO_{\tX}, \quad \forall i\geq0\,.$$
  Since $X$ is affine, we have
  $$R^i\varphi_*\OO_{\tX}=H^i(\tX,\OO_{\tX})^{\sim}, \quad \forall i\geq0\,,$$
  see \cite[Ch. III, Prop. 8.5]{Har77}. Letting
  $\tA=\tA[\DD]=\bigoplus_{u\in\sigma_M^\vee} \OO_Y(\DD(u))$ we let
  $\pi$ be the affine morphism $\pi:\tX=\SPEC_Y\tA\rightarrow
  Y$. Since the morphism $\pi$ is affine, we have
  $$H^i(\tX,\OO_{\tX})= H^i(Y,\tA)= \bigoplus_{u\in\sigma_M^\vee}
  H^i(Y,\OO_Y(\DD(u))), \quad \forall i\geq0$$ by \cite[Ch III,
  Ex. 4.1]{Har77}, proving the Theorem.
\end{proof}

As an immediate consequence of Theorem \ref{Tdirect}, in the following theorem, we characterize T-varieties having rational singularities.

\begin{theorem} \label{Trat}
Let $X=X[\DD]$, where $\DD$ is an SNC polyhedral divisor on $Y$. Then $X$ has rational singularities if and only if for every $u\in\sigma_M^\vee$
$$H^i(Y,\OO_Y(\DD(u)))=0, \quad \forall i\in \{1,\ldots,\dim Y\}\,.$$
\end{theorem}

\begin{proof}
Since $X$ is normal, by Theorem \ref{Tdirect} we only have to prove that
$$\bigoplus_{u\in\sigma_M^\vee} H^i(Y,\OO_Y(\DD(u)))=0,\quad \forall i>0$$
This direct sum is trivial if and only if each summand is. Hence $X$ has rational singularities if and only if $H^i(Y,\OO_Y(\DD(u)))=0$, for all $i>0$ and all $u\in\sigma_M^\vee$.

Finally, $H^i(Y,\FF)=0$, for all $i>\dim Y$ and for any sheaf $\FF$, see \cite[Ch III, Th. 2.7]{Har77}. Now the Lemma follows.
\end{proof}

In particular, we have the following corollary.

\begin{corollary} \label{acyc}
Let $X=X[\DD]$ for some SNC polyhedral divisor $\DD$ on $Y$. If $X$ has only rational singularities, then the structure sheaf $\OO_Y$ is acyclic i.e., $H^i(Y,\OO_Y)=0$ for all $i>0$.
\end{corollary}
\begin{proof}
This is the ``only if'' part of Theorem \ref{Trat} for $u=0$.
\end{proof}

Recall that a local ring is Cohen-Macaulay if its Krull dimension is equal to its depth. A variety is Cohen-Macaulay if all its local rings are. The following lemma is well known, see for instance \cite[page 50]{KKMS73}.

\begin{lemma} \label{kempf} %
  Let $\psi:W\rightarrow X$ be a desingularization of $X$. Then $X$
  has rational singularities if and only if $X$ is Cohen-Macaulay and
  $\psi_*\omega_W\simeq\omega_X$ 
\end{lemma}

As in Lemma \ref{rati}, applying the Leray spectral sequence shows
that the previous Lemma is still valid if we allow $W$ to have rational
singularities. In the next proposition, we give a partial criterion as
to when a T-variety is Cohen-Macaulay.

\begin{proposition} \label{cmgen}
Let $X=X[\DD]$, where $\DD$ is a proper polyhedral divisor on $Y$. If $\D^\times=\sigma(1)$, and $\D^\times_Z=\DD_Z^{(0)}$, for all prime divisor $Z\in Y$, then $X$ is Cohen-Macaulay if and only if $X$ has rational singularities.
\end{proposition}
\begin{proof}
By Theorem~\ref{th:divisor}, the contraction $\varphi:\tX\rightarrow X$ is an isomorphism in codimension 1. Thus $\varphi_*\omega_{\tX}\simeq \omega_X$. The result now follows from Lemma \ref{kempf}.
\end{proof}

For isolated singularities we can give a full classification whenever $\rank N\geq 2$.

\begin{corollary} \label{CM-iso}
Let $X=X[\DD]$, where $\DD$ is an SNC polyhedral divisor on $Y$. If $\rank N\geq 2$ and $X$ has only isolated singularities, then $X$ is Cohen-Macaulay if and only if $X$ has rational singularities.
\end{corollary}
\begin{proof}
  We only have to prove the ``only if'' part. Assume that $X$ is
  Cohen-Macaulay and let $\psi:W\rightarrow X$ be a resolution of
  singularities. Since $X$ has only isolated singularities we have
  that $R^i\psi_*\OO_W$ vanishes except possibly for $i=\dim X-1$, see
  \cite[Lemma 3.3]{Kov99}. Now Theorem \ref{Tdirect} shows that
  $R^i\psi_*\OO_W$ vanishes also for $i=\dim X-1$ since $\dim Y=\dim X
  -\rank M$ and $\rank M\geq 2$.
\end{proof}

\begin{remark}
  In \cite{Wat81} a criterion of $X$ to be Cohen-Macaulay is given in
  the case where $\rank M=1$. In this particular case, a partial
  criterion for $X$ to have rational singularities is given.
\end{remark}

\section{Canonical divisors and discrepancies}
\label{sec:canonical-div}


In the following we will restrict to the case that 
$Y$ is projective and $\sigma$ has maximal dimension. This corresponds to the fact, that there is a unique fixed point lying in the closure of all other orbits.
In particular, there is an embedding $\CC^* \hookrightarrow T$ inducing a good $\CC^*$-action on $X$. Hence, the singularity at the vertex is quasihomogeneous.

We fix isomorphisms $N=\ZZ^n$ and  $\num Y:=\wdiv Y/\equivn \cong \ZZ^r$. While we  write elements of $N$ as row vectors we will write elements of $M=N^*$ as well as elements of $\num Y$ as column vectors. 


\begin{lemma}[Prop.~3.1, \cite{PeSu08}] \label{picard}
If $\sigma$ is full dimensional and $Y$ projective then every invariant Cartier divisor on $X[D]$ is principal.
\end{lemma}

\begin{theorem}[\cite{PeSu08}, Cor. 3.15]\label{divclass}
The divisor class group of $X[\D]$  is isomorphic to 
\[\cl Y \oplus \bigoplus_\rho \ZZ D_\rho \oplus \bigoplus_{Z,v} \ZZ D_{Z,v}\]
modulo the relations
\begin{eqnarray*}
[Z] &=& \sum_{v\in \D_Z^\times} \mu(v)D_{Z,v}\,,\\
 0  &=& \sum_{\rho}  \langle u,\rho \rangle D_\rho  + \sum_{Z,v} \mu(v) \langle u,v \rangle D_{Z,v}\,.
\end{eqnarray*}
\end{theorem}
\begin{corollary}[\cite{PeSu08}, Cor. 3.15]
\label{sec:cor-Q-fac}
  $X[\D]$ is $\QQ$-factorial if and only if
  \[
  \rk \cl Y + \sum_Z (\#\D_z^\times -1) + \#\D^\times = \dim N.
  \]
  In particular $Y$ has to have a finitely generated class group.
\end{corollary}
\begin{proof}
  For $Z \notin \supp \D$ we have relations of the form $Z =
  \D_{Z,v}$, so after dividing them out we remain with finitely many
  generators and finitely many equations. Lemma~\ref{sec:lem-max-rank}
  given below shows that the relations are linearly independent and
  that there at most $\rk\cl Y +\sum_{Z \in \supp \D} \# \D^\times_Z +
  \# \D^\times$ of them.
\end{proof}

Fix a canonical divisor $K_Y = \sum_Z b_Z \cdot Z$ on $Y$. Then by
\cite{PeSu08} a $\TT$-invariant canonical divisor on $\tX[\D]$ and
$X[\D]$, respectively, is given by
\begin{align}
  \label{eq:KX}
   K_{\tX}&=q^* K_Y + \sum_{v} (\mu(v)-1)D_{v} - \sum_{\rho}E_\rho, \\
   K_X &=  \pi^*K_Y + \sum_{v} (\mu(v)-1)D_{v} - \sum_{\rho}E_\rho \nonumber
 \end{align}
 Here, the sums in the first formula run over all rays and vertices
 and in the second only over the extremal ones.

 Since $X[\D]$ has an attractive point, its Picard group is
 trivial. Hence, it is $\QQ$-Gorenstein of index $\ell$ iff we find a
 character $u \in M$ and a principal divisor $\divisor(f) = \sum_Z a_Z
 \cdot Z$ on $Y$. Such that $\divisor(f \cdot \chi^u)=K_X$. This
 observation gives rise to the following system of linear equations.
 Here, we assume that $\supp \D \cup \supp K_Y = \{Z_1, \ldots, Z_s\}$
 and $\D_{Z_i}^\times = \{v_i^1, \ldots v_i^{r_i}\}$. We set
 $\mu^i_j=\mu(v^i_j)$ and denote the classes of $Z_i \in \num Y \cong
 \ZZ^r$ by $\overline{Z}_i$.
\begin{align} \label{the-monster}
\begin{pmatrix}
  \overline{Z}_1 & \overline{Z}_2 & \ldots &  \overline{Z}_s& 0 \\
  \hline
  \mu_{1}^1 & 0 & \ldots & 0 & \mu_{1}^1 v_{1}^1 \\
  \vdots & \vdots  & & \vdots &\vdots \\
  \mu_{1}^{r_1}& 0 & \ldots & 0 & \mu_{1}^{r_1} v_{1}^{r_1} \\
         &         & \ddots & & \\
  0 & 0 & \ldots &  \mu_{s}^1    &  \mu_{s}^1 v_{s}^{1} \\
  \vdots & \vdots  & & \vdots &\vdots \\
  0 & 0 & \ldots &  \mu_{s}^{r_s} & \mu_{s}^{r_s} v_{s}^{r_s} \\
  \hline
  0 & 0 & \ldots & 0 & n_{\varrho_1} \\
  \vdots & \vdots  &  & \vdots &\vdots \\
  0 & 0 & \ldots & 0 & n_{\varrho_{r}}   
\end{pmatrix}\cdot
\begin{pmatrix}
  a_1 \\
  \vdots \\
  a_s \\
  u
\end{pmatrix} =
\begin{pmatrix}
0\\
\hline
\mu^1_1b_1+\mu^1_1-1\\
\vdots \\
\mu^{r_1}_1b_1+\mu^{r_1}_1-1\\
\vspace{-2mm}\\
\vdots\\
\mu^{1}_sb_s+\mu^{1}_s-1\\
\vdots\\
\mu^{r_s}_sb_s+\mu^{r_s}_s-1\\
\hline
-1\vspace{-2mm}\\
\vdots\\
-1
\end{pmatrix}
\end{align}

\begin{proposition}
\label{sec:prop-calc-K}
  $X[\D]$ is $\QQ$-Gorenstein of index $\ell$ if and only if the above system has a (unique) solution $u \in \frac{1}{\ell} M$ and $\ell \cdot \sum_{i=1}^s a_i \cdot Z_i$ is principal.
\end{proposition}

Note that the condition for $\QQ$-factoriality in
Corollary~\ref{sec:cor-Q-fac} is equivalent to the fact that $\cl Y$ has finite rank and the
matrix has square format. Moreover, for factoriality we get the
following stronger condition.

\begin{proposition}
  \label{sec:prop-factorial}
 $X[\D]$ is factorial iff $\cl(Y) \cong \ZZ^\ell$ and the above matrix
 is square and has determinant $\pm 1$.
\end{proposition}

\begin{proof}[Proofs]
  The propositions directly follow from the above considerations if
  the matrix is of maximal rank. This is indeed the case as
  we show in Lemma~\ref{sec:lem-max-rank} below.
\end{proof}

\begin{lemma}
\label{sec:lem-max-rank}
The columns of the matrix $A$ are linearly independent.
\end{lemma}
\begin{proof}
  We choose a non-extremal ray $\rho \in \tail \D$ and a maximal
  cone $\delta$ from the normal quasifan of $\D$, such that $\rho^\perp \cap \delta$ is a  facet, and we denote this facet by $\tau$.

  We have a linear map $F:u \mapsto \overline{\D(u)} \in \QQ^r \cong
  \numq(Y)$. Now we choose any interior element $w \in \relint
  \delta$, hence $\D(u)$ is big by the properness of $\D$. We consider
  the subspaces
\[  V:= V_1 + V_2 \qquad V_1 := \lin{\overline{Z} \mid \D(w)|_{Z} \text{ is not big}} \qquad V_2 := \lin{F(\tau)}.
\]

We claim that $F(w) \notin V$. The semi-ample and big divisor $\D(w)$ defines a birational morphism 
\[\varphi:Y \rightarrow \proj \bigoplus_{i \geq 0} H^0(Y,i\cdot \D(w)).\]
By definition $\varphi_*\D(u)$ is ample, hence big and $\varphi$ contracts every prime divisor $Z$, such that $\D(u)|_Z$ is not big.

Let us assume that $\overline{\D(w)} \in V$. It follows that \[\varphi_* \overline{D(w)} \in \varphi_*(V)=\varphi_*(V_1)+\varphi_*(V_2)=0+\varphi_* V_2.\]
But since $V_2$ does not contain any big class the same is true for $\varphi_*V_2$, but this contradicts the ampleness of $\varphi_*\D(w)$.

Now we choose a basis $B$ of $V$ and complement $\{\overline{\D(w)}\} \cup B$ to get a basis of $\numq Y$. This leads to a coordinate map $x_1:\numq \rightarrow \QQ$ corresponding to the basis element $\overline{\D(w)}$.  For every $Z_i$ there is a vertex $v_i \in \D_{Z_i}$ such that $\bangle{w,\cdot}$ is minimized at this vertex. Now we sum up the corresponding rows in the matrix with multiplicity $\nicefrac{x_1(Z_i)}{\mu(v_i)}$ (by choice of the matrix all non-extremal vertices $v_i$ have $x_1(Z_i)=0$) and get
\[
\textstyle \left(x_1(\overline{Z_1}),\ldots, x_1(\overline{Z_r}),v_\rho \right).
\]

Where, $v_\rho:=\sum_i x_1(\overline{Z_i})\cdot v_i$. By construction we have $x_1(F(u)) = \bangle{u,v_\rho}$ for $u \in \delta$. Since $x_1(F(u)) = 0$ and $x_1(F(u+\alpha w)) = \alpha$ for $u \in \tau$ and $\alpha >0$, it follows that $v_\rho \in \relint \rho$.

Now assume that $\sum_i \lambda_i c_i = 0$, where $c_i$ are the columns of the matrix. Then for every extremal ray $\rho$ we get $\sum_{i=1}^n  \lambda_{r+i} \cdot(n_\rho)_i =0$, where $(n_\rho)_i$ denotes the $i$-th coordinate of the primitive generator of $\rho$. 
Since $\sum_{i=1}^r \lambda_i \cdot \overline{Z}_i = 0$ holds because
of the first rows of the matrix we get $\sum_{i=1}^n  \lambda_{r+i}
\cdot(v_\rho)_i =0$ for every non-extremal ray of $\D$. The fact that
the tail cone $\tail \D$ has maximal dimension implies that
$\lambda_{r+1},\ldots,\lambda_{r+n}$ are zero. 

Let us assume that the first $r'$ columns correspond to prime divisors
with $\D^\times_Z \neq \emptyset$. By construction of the matrix these
columns have staircase structure. Hence,  the coefficients
$\lambda_1,\ldots, \lambda_{r'}$ vanish. The remaining columns are  of
the form $\overline{Z}_i \choose 0$. Since
the sets of exceptional vertices $\D^\times_{Z_i}$ are empty, $\D(u)|_{Z_i}$
is not big for every $u \in \relint \sigma^\vee$. Hence, the $Z_i$ are exceptional prime divisor of the birational projective map
\[
\vartheta_u: Y \rightarrow \proj \left( \bigoplus_{j \geq 0} \Gamma\left(Y,
  \CO(j \cdot \D(u))\right)\right).
\]
In particular, their images in $\num Y$ are linearly independent,
which completes the proof.
\end{proof}

Let us assume that $X$ is $\QQ$-Gorenstein. Remember that, for a birational proper morphism $r:\tX \rightarrow
X$, we have a canonical divisor $K_\tX$ on $\tX$ such that
$\text{Discr}(r)=K_\tX - r^*K_X$ is supported only at the exceptional
divisor $\sum_i E_i$. The coefficient of $\text{Discr}_r$ are called
discrepancies of $r$. The discrepancies of a pair $(X,B)$, consisting
of a normal variety and $\QQ$-Cartier divisor, are the coefficients of
$\text{Discr}(r,B):=K_\tX - r^*(K_X + B)$. With this
notation we have
\[\text{Discr}(r'\circ r) = \text{Discr}(r',-\text{Discr}(r)).\]

Consider an SNC polyhedral divisor $\D$. Fix $y \in Y$ and prime divisors $Z_1, \ldots, Z_n$ intersecting transversally at $y$ and containing the prime divisors of the support of $\varphi^* \D$ meeting $y$. From section~\ref{sec:can-desing} we know that the formal neighborhood of every fiber $\tX_y$ of $\tX[\D] \rightarrow Y$ is isomorphic to that of of a closed subset of a toric variety corresponding to some cone $\sigma_y' \in N_\QQ \otimes \QQ^{\dim Y}$. Moreover, the isomorphism identifies  $\widetilde{D}_{Z_i,v}$ and 
$V(\QQ_{\geq 0}(v,e_{i}))$ as well as  $E_\rho$ and $V(\rho \times \mathbf{0})$. 

Now we may calculate a representation $K_X= \pi^*H + \divisor(\chi^u)$  of the canonical divisor on $X$ by solving a system of linear equations as in proposition~\ref{sec:prop-calc-K}. Here, $H = \sum_Z a_Z \cdot Z$ is a principal divisor on $Y$. Having such a representation we get the discrepancies of 
$\tX[\D] \rightarrow X[\D]$ at $\widetilde{D}_{Z,v}$ or $\widetilde{E}_\rho$, respectively as
\begin{equation}
  \label{eq:discr}
  \text{discr}_{Z,v}=\mu(v)(b_Z - a_Z -\langle u,v \rangle + 1) -1 ,\qquad \text{discr}_{\rho}=-1-\langle u,n_\rho \rangle.
\end{equation} 

We may also consider a toroidal desingularization
$\varphi:\overline{X} \rightarrow \tX[\D]$, obtained by toric
desingularisations of the $X_{\sigma_y'}$. Since the discrepancies
$\text{discr}_{Z,v}$ vanish for $Z \notin \supp \D$, the discrepancies
divisor on $\tX[\D]$ corresponds to a toric divisor $B \subset X_{\sigma_y'}$ and we are able to calculate the discrepancy divisor $\text{Discr}(\varphi,B)$ by toric methods.

We say that $X$ or $(X,B)$, respectively, is log-terminal, if for a desingularization the discrepancies are
$>-1$. It's an easily checked fact that a toric pair $(X_\sigma,B)$ is log terminal as long as $B < \sum_\rho V(\rho)$ and $B-\sum_\rho V(\rho)$ is $\QQ$-Cartier. We may argue as in the proof of \cite[Lem.~5.1]{MR2017225}: Since $K_{X_\sigma} = -\sum_\rho V(\rho)$ the $\QQ$-divisor $B+K_{X_\sigma}$ corresponds to an element $u \in M_\QQ$ such that $\langle u, n_\rho\rangle <0$ for every $\rho \in \sigma(1)$. But then the primitive generator $n_{\rho'}$ of a ray $\rho'$ in a subdivision $\Sigma$ of $\sigma$ is a positive combinationation of primitive generators $n_\rho$ of rays of $\sigma$. Hence,  $\langle u, n_{\rho'}\rangle <0$ holds. But now we have $\text{discr}_{V(\rho)}=-1-\langle u, n_{\rho'} \rangle > -1$.

A $\QQ$-divisor $B=\sum_{Z} b_Z \cdot Z$ is called a boundary divisor if
for the coefficients we have $0 < b_Z \leq 1$. For a polyhedral
divisor on $Y$ we define the boundary divisor $B:=\sum_Z
\frac{\mu_Z-1}{\mu_Z}\cdot Z$ on $Y^\circ := Y \setminus \bigcup_{Z,
  \D_Z^\times = \emptyset} Z$, where $\mu_Z$ is defined as $\max
\{\mu(v) \mid v \in \D^\times_Z\}$.

\begin{theorem}
\label{sec:thm-log-terminal}\ 
\begin{enumerate}
\item 
  If $X[\D]$ is log-terminal then 
there exists a boundary divisor $B' \geq B$,
such that $(Y^\circ, B')$ is weakly log-Fano (but not necessarily complete). In particular, $(Y^\circ, B)$ is a log-terminal pair and $-(K_{Y^\circ} +B)$ is big.
\item \label{item:cample}
Assume that $\D$ is completely ample and $X[\D]$ is $\QQ$-Gorenstein, then $X[\D]$ is log-terminal if and only if $(Y,B)$ is log-Fano.
\end{enumerate}
\end{theorem}
\begin{proof}
  Let us choose a representation $\D=\sum_{Z \in I} \Delta_Z \cdot Z$, with $I$ being a finite set of Cartier divisors on $Y$.

Let $K_X = \pi^*H +  \divisor(\chi^w)$ a representation as above. By (\ref{eq:KX}) we have 
\begin{equation}
  \label{eq:boundary-formula}
  K_Y + B  = \pi^*H + \sum_Z \langle w, v_z \rangle\cdot Z,
\end{equation}
here $v_Z \in \D_Z^\times$ denotes the vertex where $\mu$ obtains its maximum. 

For any ray $\rho \in \sigma(1)$ the value $\langle w, n_\rho \rangle$ has to be negative because of the condition $\text{discr}_\rho=-1-\langle w, n_\rho \rangle > -1$ for non-extremal rays or $\langle w, n_\rho \rangle = -1$ for extremal rays, respectively. It follows that $-w \in \relint(\sigma^\vee)$.

Now we consider the divisor $B' > B$ defined by
\begin{equation}
  \label{eq:boundary-formula-prime}
  K_Y + B'  = \pi^*H - \D(-w) = \pi^*H + \sum_Z \langle w, v'_z \rangle\cdot Z,
\end{equation}
here $v'_Z$ denotes the vertices in $\D_Z$, where $-w$ is minimized.
Since $D(-w)$ is semi-ample and big this implies the weak Fano property  for the pair $(Y^\circ,B')$.



Now consider a birational proper morphism $\varphi: \widetilde{Y}
\rightarrow Y$.
And denote $\tX[\varphi^*\D]$ by $\tX$. Consider a prime divisor $E
\subset \widetilde{Y}$ and denote by $(\varphi^*Z)_E$ the coefficient
of $\varphi^*Z$ at $E$. Note that $v'_E:=\sum_Z (\varphi^*Z)_E \cdot
v'_Z$ is a vertex in $(\varphi^*\D)_E$.  If $v'_E$ is non-extremal, by
(\ref{eq:discr}) we get for the discrepancy $\text{discr}_{v'_E}$
\begin{align}
  \text{discr}_{v'_E} &= \mu(v'_E)\left((K_{\widetilde{Y}})_E-
(\varphi^*H)_E - \langle w, v'_E \rangle + 1\right) -1 \label{eq:discr-vE} \\
{ } &= \mu(v'_E)\left((K_{\widetilde{Y}})_E- \varphi^*(K_Y+B')_E + 1 \right)-1. \nonumber
\end{align}
For the case that $E$ is an exceptional divisor of $\varphi$ this proves the log-terminal property for $(Y^\circ,B')$. If $\varphi=\text{id}$ and $E=Z$ is a prime divisor of $Y$ and $v_E'=v_Z'$ is non-extremal we obtain $B'_Z < 1$. But, for $v'_Z$ extremal we have $B'_Z = \frac{\mu(v_Z')-1}{\mu(v'_Z)} <1$ by (\ref{eq:KX}). Hence, $B'$ is indeed a boundary divisor.

If $\D$ is completely ample all vertices are extremal. Hence, we have
$Y^\circ = Y$ and $B'=B$. Since, $\D(-w)$ is even ample this proves
one direction of (\ref{item:cample}) in the Theorem.

For the other direction we first show that the Fano property for
$(Y,B)$ implies that $-w \in \sigma^\vee$. For extremal rays $\rho \in
\D^\times$ we have $\langle w, n_\rho \rangle= -1$ by (\ref{eq:discr}).  For a non-extremal ray $\rho$ we consider a maximal
chamber of linearity $\delta \subset \sigma^\vee$ such that $\tau
=\rho^\perp \cap \delta$ is a facet. This corresponds to a family of
vertices $v^u_Z$ such that $\D(u) = \sum_Z \langle u, v_Z^u\rangle
\cdot Z$ for $u \in \relint \delta$. Now there exists a decomposition
$-w = \alpha \cdot u + u_\tau$ such that $u_\tau \in \tau$ and $u \in
\relint \delta$.  Hence, we have
\[
-K_Y-B \;\leq\; -\pi^*H + \sum_Z \langle -w, v_Z^u \rangle \cdot Z
\;\sim\; \D(u_\tau) + \alpha \D(u).
\]
By our precondition $-K_Y-B$ is big. This implies that the right hand side is big, too. Then we must have $\alpha > 0$, since $\D(u)$ is big, but $\D(u_\tau)$ is not. By $\langle -w, n_\rho \rangle = \alpha \cdot \langle u, n_\rho \rangle$
we conclude that $\langle -w, n_\rho \rangle > 0$ and hence 
$\text{discr}_\rho = -1 - \langle w, n_\rho \rangle> -1$ and $-w \in \sigma^\vee$.  Let $\varphi:\widetilde{Y} \rightarrow Y$ be a desingularization such that $\varphi^* \D$ is SNC. By equation (\ref{eq:discr-vE}) we infer that $\text{disc}_{v_E} > -1$ for every exceptional divisor $E$ and every  vertex $v_E \in (\varphi^*\D)_E$. By lemma~\ref{sec:lem-log-terminal} this completes the proof.
\end{proof}

\begin{remark}
Let $\D$ be a proper polyhedral divisor on some projective variety
$Y$. If  $Y$ is a curve, or more general $Y$ has Picard rank one or
$\rank N=1$ the $\D$ is always completely ample. In these cases we
obtain a sufficient and necessary criterion for log-terminality.

As a special case of the theorem we recover the well known fact, that the log-terminal property a section ring characterizes log-Fano varieties.
\end{remark}

\begin{lemma} \label{sec:lem-log-terminal} Let $\D$ be an SNC
  polyhedral divisor. Then $X[\D]$ is log-terminal iff the
  discrepancies of $\psi: \tX[\D] \rightarrow X[\D]$ are all greater
  than $-1$.
\end{lemma}
\begin{proof}
  The discrepancies of a toric map $X_{\Sigma} \rightarrow X_{\sigma}$ and an invariant boundary divisor $B < \sum_{\rho \in \sigma(1)} V(\rho)$ are greater than $-1$. By the toroidal structure of $\tX[\D]$ and the considerations above this implies that the discrepancies of a resolution $\varphi \circ\psi$ obtained as the composition of a toroidal resolution $\varphi$ of $\tX$ and $\psi$ has discrepancies $>-1$.
\end{proof}

\begin{corollary}
  Every $\QQ$-Gorenstein T-variety $X$ of complexity $c$ with singular locus of codimension $> c+1$ is log terminal.
\end{corollary}

\begin{proof} We may assume that $X$ is affine. Given a SNC polyhedral
  divisor for $X$ we consider an exceptional divisors
  $\widetilde{D}_{Z,v},\;\widetilde{E}_{\rho} \subset \tX[\D]
  \rightarrow X$ with discrepancy $\leq -1$.  By the orbit
  decomposition of $X[\D]$ given in \cite{MR2207875} we know that
  $\widetilde{E}_{\rho}$ is contracted via $r$ to a closed subvariety
  of codimension at most $c+1$ in $X[\D]$ and $\widetilde{D}_{Z,v}$ to
  a subvariety of codimension at most $c$. But
  $r(\widetilde{E}_{\rho})$, $r(\widetilde{D}_{Z,v}) $ are necessarily part of the singular locus. 
\end{proof}


\section{Complexity one} \label{sec:complexity-one}

As an application, in this section we restate our previous results in this particular setting. This allows us to rediscover some well known results with our methods.

Let $\DD$ be a proper polyhedral divisor on $Y$. If the corresponding
$\TT$-action on $X=X[\DD]$ has complexity 1 then $Y$ is a curve. Since
any normal curve is smooth and any birational morphism between smooth
curves is an isomorphism we have that the base curve $Y$ is uniquely
determined by the $\TT$-action on $X$.

Furthermore, any curve $Y$ is either affine or projective, and any polyhedral divisor $\DD$ on $Y$ is SNC and completely ample. Let $\DD$ and $\DD'$ be two proper polyhedral divisors on $Y$. Then $X[\DD]\simeq X[\DD']$ equivariantly if and only if the application $$\Delta:\sigma^\vee\rightarrow \operatorname{Div}_\QQ(Y),\qquad  u\mapsto \DD(u)-\DD'(u)$$
 is the restriction of a linear map and $\Delta(u)$ is principal for all $u\in\sigma_M^\vee$.

The simplest case is the one where $N=\ZZ$ i.e., the case of $\KK^*$-surfaces. In this particularly simple setting there are only two non-equivalent pointed polyhedral cones in $N_\QQ\simeq\QQ$ corresponding to $\sigma=\{0\}$ and $\sigma=\QQ_{\geq 0}$.

Assuming further that $Y$ is projective, then $\sigma\neq \{0\}$ and so we can assume that $\sigma=\QQ_{\geq 0}$. In this case $\DD(u)=u\DD(1)$. Hence $\DD$ is completely determined by $\DD_1:=\DD(1)$, and $X[\DD]\simeq X[\DD']$ equivariantly if and only if $\DD_1-\DD_1'$ is principal. We also let
\begin{align} \label{D1-desc}
\DD_1=\sum_{i=1}^{r} \frac{e_i}{m_i}\cdot z_i,\quad\mbox{where}\quad \gcd(e_i,m_i)=1,\mbox{ and }m_i>0\,.
\end{align}
In this case, the algebra $A[\DD]$ is also known as the section ring
of $\DD_1$.

\subsection{Isolated singularities}
\begin{proposition}
\label{sec:prop-D-smooth}
  A polyhedral divisor $\D$ having a complete curve $Y$ as its locus defines a smooth variety if and only if $Y=\PP^1$,
  $\D \sim \Delta_{y} \otimes P + \Delta_{z} \otimes Q$ 
  and
  $$\delta:=\overline{\QQ^+ \cdot (\{1\} \times \Delta_y \cup \{-1\} \times \Delta_z)} \subset \QQ\times N$$ 
  is a regular cone.
\end{proposition}

\begin{proof}
  One direction is obvious, because the polyhedral divisor $\D$ describes the toric variety $X_\delta$ with a $T_N$-action induced by
  $N \hookrightarrow \QQ\times N$ \cite[section~11]{MR2207875}. For the other direction note that for $v \in \relint \sigma \cap N$ we get a positive grading 
  $$\Gamma(X,\CO_X) = \bigoplus_{i \geq 0} \bigoplus_{\langle u,v \rangle = i} \Gamma(Y, \CO(\D(u))).$$ 
  There is a full sub-lattice $M' \subset \sigma^\perp$ such that $\D(u)$ is integral for $u \in M'$. Because $\D$ is proper we get
  $\D(u)$ and $\D(-u)$ to be of degree $0$ and even principal $\D(u)=\divisor(f_u)$ and $\D(-u)=-\divisor(f_u)$ for $u \in M'$.
  This implies that $\Gamma(X,\CO_X)_0 \subset K(X)$ is generated by $\{f^{\pm 1}_{u_1} \chi^{\pm u_1}, \ldots, f^{\pm 1}_{u_r} \chi^{\pm u_r}\}$ where the $u_i$ are elements of a $M'$-basis. We get $\Gamma(X,\CO_X)_0 \cong k[M']$. 

  The following lemma implies that 
  $$X = \spec k[M'][\ZZ_{\geq 0}^{\dim \sigma}]$$
  holds and $\D$ has the claimed form.
\end{proof}

\begin{lemma}
  Let $A_0=k[\ZZ^r]$, and $A = \bigoplus_{i\geq 0} A_i$ be a finitely generated positive graded $A_0$-algebra. $A$ is a regular ring if
  and only if $A$ is freely generated as an $A_0$-algebra. 
\end{lemma}
\begin{proof}
  We may choose a minimal homogeneous generating system $g_1, \ldots, g_l$ for $A$ and consider a maximal ideal $\m_0$ of $k[\ZZ^r]$ then
  $\m := A\m_0 + (g_1,\ldots, g_l)$ is a maximal ideal of $A$. If $A$ is regular and $l > \dim A - r$ then 
$\dim_k \m/\m^2 = \dim A < r + l$ and w.l.o.g. we may assume that there is a $g_i$ such that there is a (finite) homogeneous relation 
  $g_i = \sum_{|\alpha|\geq 2} b_\alpha g^\alpha$, with $b_\alpha \in A_0$, $\alpha \in \NN^l$. For degree reasons we may assume that $b_\alpha = 0$ 
  if $\alpha_i \neq 0$. But this contradicts the minimality of $\{g_1,\ldots, g_l\}$. So we have $l = \dim A - r$, thus $A$ is free as an $A_0$-algebra.
\end{proof}

\begin{theorem}
\label{sec:thm-toric-sing}
Let $\D = \sum_z \Delta_z \cdot z$ be a polyhedral divisor on an
affine curve. Then $\tX[\D]$ is smooth if and only if for every $z \in
Y$ the cone
$$\delta_z := \overline{\QQ^+ \cdot (\{1\} \times \Delta_z)} \subset \QQ\times N$$
is regular. Moreover the singularity at $q^{-1}(z)$ is analytically isomorphic to the toric singularity $X_{\delta_z}$.
\end{theorem}
\begin{proof}
  This is a special case of Proposition~\ref{toroidal}. In our setting
  every polyhedral divisor is SNC and moreover we have $\tX[\D]=X[\D]$
  since $Y$ is affine.
\end{proof}

For a proper polyhedral divisor $\D$ on a complete curve $Y$ we consider it faces. By definition these are the polyhedral divisors 
obtained defined as follows 
\[
\D^u = \sum_z \face(\D_z,u),\quad u \in \sigma^\vee
\]
Such a face is called facet, if $\min_{z\in Y} \codim \D^u_z = 1$.

\begin{theorem}
\label{sec:thm-isolated-sing}
$X[\D]$ is an isolated singularity iff for every facet
\begin{enumerate}
\item 
for every facet with
$\deg \D^u \subsetneq \face(\sigma, u)$ the variety $X[\D^u]$ is non-singular,

\item and for every other facet $\tX[\D^u]$ is non-singular.
\end{enumerate}

\end{theorem}
\begin{proof}
  First note the the inclusion $\D^u \subset \D$ defines an open inclusion of $X(\D^u) \hookrightarrow X(\D)$ in the first case and of  $\tX(\D^u) \hookrightarrow X(\D)$ in the second case. By the orbit decomposition given in \cite{AHS08} we know that every orbit except from the unique fixed point is contained in one of these open subsets.
\end{proof}

\subsection{Rational singularities}

The following proposition gives a simple characterization of
T-varieties of complexity 1 having rational singularities.

\begin{proposition} \label{rcom1}
Let $X=X[\DD]$, where $\DD$ is an SNC polyhedral divisor on a smooth curve $Y$. Then $X$ has only rational singularities if and only if 
\begin{enumerate}[$(i)$]
 \item $Y$ is affine, or 
 \item $Y=\PP^1$ and $\deg\lfloor\DD(u)\rfloor \geq -1$ for all $u\in\sigma_M^\vee$.
\end{enumerate}
\end{proposition}
\begin{proof}
If $Y$ is affine, then the morphism $\varphi:\tX[\DD]\rightarrow X$ is an isomorphism. By Lemma \ref{toroidal} $X$ is toroidal and thus $X$ has only toric singularities and toric singularities are rational.

If $Y$ is projective of genus $g$, we have $\dim H^1(Y,\OO_Y)=g$. So by Corollary \ref{acyc} if $X$ has rational singularities then $C=\PP^1$. Furthermore, for the projective line we have $H^1(\PP^1,\OO_{\PP^1}(D))\neq 0$ if and only if $\deg D\leq-2$ \cite[Ch. III, Th 5.1]{Har77}. Now the corollary follows from Theorem \ref{Trat}.
\end{proof}

In the next proposition we provide a partial criterion for the Cohen-Macaulay property in the complexity 1 case. Recall that in the complexity one case, a ray $\rho\in\sigma(1)$ is extremal if and only if $\deg\DD\cap\rho=\emptyset$.

\begin{proposition} \label{cm1}
Let $X=X[\DD]$, where $Y$ is a smooth curve and $\DD$ is an SNC
polyhedral divisor on $Y$. The $X$ is Cohen-Macaulay if one of the following conditions hold,
\begin{enumerate}[$(i)$]
 \item $Y$ is affine, or
 \item $\rank M=1$
\end{enumerate}
Moreover, if $Y$ is projective and $\DD^\times=\sigma(1)$, then $X$ is Cohen-Macaulay if and only if $X$ has rational singularities.
\end{proposition}

\begin{proof}
If $Y$ is affine then $X=\tX[\DD]$. Thus $X$ has rational singularities and so $X$ is Cohen-Macaulay. If $\rank M=1$ then $X$ is a normal surface. By Serre S$_2$ normality criterion any normal surface is Cohen-Macaulay, see Theorem 11.5 in \cite{Eis95}. Finally, the last assertion is a specialization of Proposition \ref{cmgen}.
\end{proof}

\begin{remark}
  Corollary \ref{CM-iso} and Proposition \ref{cm1} give a full
  classification of isolated Cohen-Macaulay singularities on
  T-varieties of complexity one.
\end{remark}

\subsection{Log-terminal and canonical singularities}
In the complexity one case every proper polyhedral divisor is
completely ample, since ampleness and bigness coincide. Now,
theorem~\ref{sec:thm-log-terminal} gives rise to the following

\begin{corollary}
  Let $\D = \sum_z \Delta_z \cdot z$ be a proper polyhedral divisor on
  a curve $Y$. Assume that $X[\D]$ is $\QQ$-Gorenstein.

  Then $X[\D]$ is log-terminal if and only if either
  \begin{enumerate}
  \item $Y$ is affine, or
  \item $Y=\PP^1$ and $\sum_z \frac{\mu_z-1}{\mu_z} < 2$.
  \end{enumerate}
\end{corollary}
\begin{proof}
  By theorem~\ref{sec:thm-log-terminal} we know that
  $-K_Y-\sum_z\frac{\mu_z-1}{\mu_z} \cdot z$ has to be ample. This is
  the case exactly under the conditions on the corollary.
\end{proof}

\begin{remark}\ 
  \begin{enumerate}
  \item The second condition in the corollary can be made more
    explicit: there are at most three coefficients $\D_{z_1}$,
    $\D_{z_2}$, $\D_{z_3}$ on $\PP^1$ having non-integral vertices,
    and the triple $(\mu_{z_1}, \mu_{z_2}, \mu_{z_3})$ is one of the
    Platonic triples $(1,p,q)$, $(2,2,p)$, $(2,3,3)$, $(2,3,4)$, and
    $(2,3,5)$.

  \item It is well known that log-terminal singularities are
    rational. Indeed, since $\tfrac{a}{b}-\left\lfloor
      \tfrac{a}{b}\right\rfloor\leq \tfrac{b-1}{b}$, the condition
    $\sum_z \frac{\mu_z-1}{\mu_z}<2$ ensures that
    $\deg\lfloor\DD(u)\rfloor > \deg\DD(u)-2\geq-2$. Thus $X[\DD]$ has
    rational singularities by Corollary \ref{rcom1}.
  \end{enumerate}
\end{remark}

As a direct consequence we get the following corollary characterizing
quasihomogeneous surfaces having log-terminal singularities. Recall
the definition of $\DD_1$ in \eqref{D1-desc}.
\begin{corollary}\label{sec:cor-log-term-2}
  Every quasihomogeneous log-terminal surface singularity is
  isomorphic to the section ring of the following divisor on
  $Y=\PP^1$.
  $$\DD_1=\frac{e_1}{m_1}\cdot[0]+\frac{e_2}{m_2}\cdot[1]+\frac{e_3}{m_3}\cdot[\infty],
  \quad \mbox{and}\quad \deg\DD_1>0\,.$$ Here $(m_1,m_2,m_3)$ is one
  of the Platonic triples $(1,p,q)$, $(2,2,p)$, $(2,3,3)$, $(2,3,4)$,
  and $(2,3,5)$, where $p\geq q\geq 1$, and $r\geq 2$.
\end{corollary}

We now characterize quasihomogeneous surfaces having canonical singularities.

\begin{theorem}
  Every quasihomogeneous canonical surface singularity is isomorphic
  to the section ring of one of the following $\QQ$-divisors on
  $\PP^1$:
  $$
  \begin{tabular}[htbp]{lll}
    $A_i:\ $ & $\dfrac{i+1}{i}\cdot[\infty]\,,$ & $i \geq 1\,.$\vspace{2mm}\\
    $D_i: $ & $\dfrac{1}{2}\cdot[0]+\dfrac{1}{2}\cdot[1]
      -\dfrac{1}{(i-2)}\cdot[\infty]\,,$ & $i \geq 4\,.$\vspace{2mm}\\
    $E_i: $ & $\dfrac{1}{2}\cdot[0]+ \dfrac{1}{3}\cdot[1]
      -\dfrac{1}{(i-3)}\cdot[\infty]\,,$  &$i=6,7,8\,.$\\
  \end{tabular}
  $$
\end{theorem}
\begin{proof}
  Canonical implies log-terminal. Hence, it suffices to consider a
  polyhedral divisors $\DD$ on $\PP^1$ as in Corollary
  \ref{sec:cor-log-term-2} i.e., of the form
  $$\DD_1=\frac{e_1}{m_1}\cdot[0]+\frac{e_2}{m_2}\cdot[1]+\frac{e_3}{m_3}\cdot[\infty],
  \quad \mbox{and}\quad \deg\DD_1>0\,.$$
  
  Let $1 \leq m_1 \leq m_2 \leq m_3$. Up to linear equivalence we may
  assume that $m_1>e_1 \geq 0$ and $m_2> e_2 \geq 0$. If $m_1=1$ we
  have $e_1=0$ and $X$ is isomorphic to the affine toric variety given
  by the cone $\pos((e_2,m_2), (e_3,-m_3))$. But every cone is
  isomorphic to a subcone of $\pos((0,1), (1,1))$. Therefor we may
  assume that $m_1=m_2=1$, $e_1=e_2=0$ and $e_3 \geq m_3$.

  The system of equations from proposition~\ref{sec:prop-calc-K} takes
  the form
  \[
  \begin{array}{cccc|c}
    1 & 1 & 1 & 0   & 0\\
    m_1 & 0 & 0 & e_1 & m_1-3\\
    0 & m_2 & 0 & e_2 &  m_2-1\\
    0 & 0 & m_3 & e_3 &  m_3-1
  \end{array}
  \]

  Any solution $(a_1,a_2,a_3,u)$ must also fulfill
  \begin{equation}
    \label{eq:deg-boundary}
    u \cdot \deg D = \sum_i \frac{m_i-1}{m_i}-2.
  \end{equation}

  The formula for the discrepancy at $E_\rho$ yields $\text{disc}_\rho
  = -1-u$. Hence, we need $u < -1$. For the case $(1,1,q)$
  equation~(\ref{eq:deg-boundary}) yields $u =
  -\frac{m_3+1}{e_3}$. Hence we must have $e_3=m_3+1$. For the cases
  $(2,2,r)$ equation~(\ref{eq:deg-boundary}) takes the form $u \cdot
  \frac{m_3+e_3}{m_3} = \frac{1}{m_3}$ and we get $e_3=1-m_3$. For the
  remaining cases $(2,3,r)$ we get
  \[ \frac{3+2e_2+2e_3}{6} = \frac{1}{6}, \qquad
  \frac{6+4e_2+3e_3}{12} = \frac{1}{12}, \qquad
  \frac{15+10e_2+6e_3}{30} = \frac{1}{30}.\] Since $1,2$ are the only
  options for $e_2$ we infer that $e_2=1$ and $e_3=1-m_3$.
\end{proof}

\subsection{Elliptic singularities}

Let $(X,x)$ be a normal  singularity, and let $\psi:W\rightarrow X$ be a resolution of the singularity $(X,x)$. One says that $(X,x)$ is an \emph{elliptic singularity}\footnote{Some authors call such $(X,x)$ a strongly elliptic singularity.} if 
$$R^i\psi\OO_W=0\, \forall i\in\{1,\ldots, \dim X-1\},\quad \mbox{and}\quad R^{\dim X-1}\psi_*\OO_W\simeq \KK\,.$$
An elliptic singularity is \emph{minimal} if it is Gorenstein. e.g., \cite{Lau77}, and \cite{Dai02}.

In the complexity one case $R^i\psi_*\OO_W=0$, for all $i\geq 2$. Thus, the only way to have elliptic singularities is to have $M=\ZZ$. That is, the case of $\KK^*$-surfaces. In the following we restrict to this case.

We give now a simple criterion as to when $X[\DD]$ is $\QQ$-Gorenstein. This is a specialization of \ref{sec:prop-calc-K} Recall that the boundary divisor is defined in this particular case as $B=\sum_i\tfrac{m_i-1}{m_i}\cdot z_i$, see \eqref{D1-desc}. We let $u_0=\nicefrac{\deg(K_Y+B)}{\deg(\DD_1)}$.

\begin{lemma} \label{gor}
The surface $X[\DD]$ is $\QQ$-Gorenstein of index $\ell$ if and only if $u_0\in\tfrac{1}{\ell}\ZZ$ and the divisor $\ell(u_0\DD_1-K_Y-B)$ is principal. Furthermore, if $X[\DD]$ is $\QQ$-Gorenstein of index 1 then $X[\DD]$ is Gorenstein.
\end{lemma}
\begin{proof}
Let a canonical divisor of the curve $Y$ be given by
$$K_Y=\sum_{i=r+1}^k b_iz_i, \qquad \mbox{where} \qquad z_i\neq z_j,\ \forall i\neq j\,.$$
With the notation of Proposition \ref{sec:prop-calc-K} we have that $D_{z_i}^\times=\{\nicefrac{e_i}{m_i}\}$ for $i\leq r$ and $D_{z_i}^\times=\{0\}$ otherwise. Furthermore $\mu_i=m_i$ and $\mu_iv_i=e_i$ for $i\leq r$, and $\mu_i=1$ and $\mu_iv_i=0$ otherwise.
With this considerations, the system of equations in
\eqref{the-monster} becomes
\begin{align*}
m_ia_i+e_iu=m_i-1,& \qquad\forall i\leq r \\
a_i=b_i,&\qquad \forall i\geq r+1\,,
\end{align*}
and so 
$$a_i=-u\frac{e_i}{m_i}+\frac{m_i-1}{m_i},\qquad\forall i\leq r\,.$$

This yields $D=-u\DD_1+B+K_Y$ and $u=u_0$. This shows the first assertion. The second one follow at once since any normal surface is Cohen-Macaulay.
\end{proof}
\begin{remark}
The $\QQ$-Gorenstein assertion of the previous lemma is true in general for affine $\KK^*$-varieties with the same proof, Cf. \cite{Wat81}.
\end{remark}

In the following theorem we characterize quasihomogeneous (minimal) elliptic singularities of surfaces.
\begin{theorem} \label{ellip}
Let $X=X[\DD]$ be a normal affine surface with an effective elliptic 1-torus action, and let $\bar{0}\in X$ be the unique fixed point. Then $(X,\bar{0})$ is an elliptic singularity if and only if one of the following two conditions hold.
\begin{enumerate}[$(i)$]
 \item $Y=\PP^1$, $\deg\lfloor u\DD_1\rfloor\geq-2$ for all $u\in\ZZ_{>0}$, and $\deg\lfloor u\DD_1\rfloor=-2$ for one and only one $u\in\ZZ_{>0}$.
 \item $Y$ is an elliptic curve, and for every $u\in\ZZ_{>0}$, the divisor $\lfloor u\DD_1\rfloor$ is not principal and $\deg\lfloor u\DD_1\rfloor\geq 0$.
\end{enumerate}
Moreover, $(X,\bar{0})$ is a minimal elliptic singularity if and only if $(i)$ or $(ii)$ holds, $u_0$ is integral and $u_0\DD_1-K_Y-B$ is principal.
\end{theorem}

\begin{proof}
Assume that $Y$ is a projective curve of genus $g$, and let $\psi:Z\rightarrow X$ be a resolution of singularities. By Theorem \ref{Tdirect},
$$R^1\psi_*\OO_Z=\bigoplus_{u\geq0}H^1(Y,\OO_Y(u\DD_1))\,.$$
Since $\dim R^1\psi_*\OO_Z\geq g=\dim H^1(Y,\OO_Y)$, if $X$ has an elliptic singularity then $g\in\{0,1\}$. 

If $Y=\PP^1$ then $(X,\bar{0})$ is an elliptic singularity if and only if $H^1(Y,\OO_Y(u\DD_1)=\KK$ for one and only one value of $u$. This is the case if and only if $(i)$ holds. If $Y$ is an elliptic curve, then $H^1(Y,\OO_Y)=\KK$. So the singularity $(X,\bar{0})$ is elliptic if and only if $H^1(Y,u\DD_1)=0$ for all $u>0$. This is the case if and only if $(ii)$ holds.

Finally, the last assertion concerning maximal elliptic singularities follows immediately form Proposition \ref{gor}.
\end{proof}

\begin{example}
By applying the criterion of Theorem \ref{ellip}, the following combinatorial data gives rational $\KK^*$-surfaces with an elliptic singularity at the only fixed point.
\begin{enumerate}[$(i)$]
 \item $Y=\PP^1$ and $\DD_1=-\tfrac{1}{4}[0]-\tfrac{1}{4}[1]+\tfrac{3}{4}[\infty]$. In this case $X=\spec A[Y,m\DD_1]$ is isomorphic to the surface in $\AF^3$ with equation
$$x_1^4x_3+x_2^3+x_3^2=0\,.$$
 \item $Y=\PP^1$ and $\DD_1=-\tfrac{2}{3}[0]-\tfrac{2}{3}[1]+\tfrac{17}{12}[\infty]$. In this case $X=\spec A[Y,m\DD_1]$ is isomorphic to the surface
$$V(x_1^4x_2x_3-x_2x_3^2+x_4^2\ ;\, x_1^5x_3-x_1x_3^2+x_2x_4\ ;\,x_2^2-x_1x_4)\subseteq \AF^4\,.$$
\end{enumerate}
This last example is not a complete intersection since otherwise
$(X,\bar{0})$ would be Gorenstein i.e., minimal elliptic which is not
the case by virtue of Theorem \ref{ellip}. In the first example the
elliptic singularities is minimal, since every normal hypersurface is
Gorenstein.
\end{example}

\section{Factorial T-varieties}\label{sec:fact-t-vari}

Let $Y$ be a normal projective variety having class group $\ZZ$. Hence, we have a canonical degree map $\cl(Y) \rightarrow \ZZ$ by sending the ample generator to $1$. We choose a set $\mathcal{Z}=\{(Z_1,\mu_i), \ldots (Z_s,\mu_s)\}$ of prime divisors of degree $1$ and corresponding tuples $\mu_i=(\mu_{i1},\ldots \mu_{ir_i}) \in \NN^{r_i}$. We assume that the integers $\gcd(\mu_i)$ are pairwise coprime and define $|\mathcal{Z}|:=\sum_i (r_i-1)$.

We give construction of a polyhedral divisor on $Y$ with 
polyhedral coefficients in $N_\QQ=\QQ^{|\mathcal{Z}|+1}$ by induction on 
$|\mathcal{Z}|$.

\begin{construction}
If $|\mathcal{Z}|=0$ the $\mu_1,\ldots, \mu_s$ are positive pairwise coprime integers. Hence, there is are integer coefficients $e_1, \ldots, e_s$ such that $1 = \sum e_i \mu_i$. Now, we define the vertices $v_{ij}:=\frac{e_i\mu_i}{\mu_1 \cdots \mu_s} \in N_\QQ$. 

If $|\mathcal{Z}|>0$ there is a $j\in\{1,\ldots,r\}$ such that 
$r_j>1$. Now, we consider the data $\mathcal{Z}'$ obtained from $\mathcal{Z}$ by replacing $\mu_j$ by \[\mu_j':=(\mu_{j 1},\ldots,\gcd(\mu_{j r_j-1},\mu_{j r_j})).\]
By induction we obtain vertices $v_{ik}'$ with $v_{jr_j-1}'$ being the vertex corresponding to $\mu_{jr_j-1}'=\gcd(\mu_{jr_j-1},\mu_{jr_j})$.
We find coefficients $\alpha,\beta \in \ZZ$ such that 
$\mu_{jr_j-1}'=\alpha \mu_{jr_j-1}+ \beta \mu_{jr_j}$. Now, we define the vertices 
\[v_{j r_j -1}=\left(v_{j 1}',-\frac{\beta}{\mu_{jr_j-1}}\right); \qquad  v_{j r_j}=\left(v_{j 1}', \frac{\alpha}{\mu_{jr_j}} \right)
\]
and $v_{ik}=(v_{ik}',0)$ for $i \neq j$ or $k < r_j-1$.

For every set of admissible data $\mathcal{Z}$ we can define a
polyhedral divisor $\D=\D(\mathcal{Z})$ on $Y$. The tail fan is
spanned by the rays $\QQ_{\geq 0} \cdot \sum_i v_{ik_i}$, where $1\leq
k_i\leq r_i$. And the vertices of $\D_{Z_i}$ are exactly the
$v_{ik}$. We denote the corresponding algebra $A[\D]$ also by
$A[\mathcal{Z}]$.
\end{construction}

\begin{theorem}
\label{sec:thm-ufd}
  $A[\mathcal{Z}]$ is a normal factorial ring. 
\end{theorem}

\begin{proof}
For $|\mathcal{Z}|=0$ the matrix of relations for the class group
has the following form
  \[
M_{\mathcal{Z}}=\begin{pmatrix}
  1 &  \ldots &  1 & 0 \\
  \hline
  \mu_{1} &  \ldots & 0 & \mu_{1}v_{11} \\
  \vdots&  \ddots & \vdots & \vdots \\
  0 &  \ldots &  \mu_{s} & \mu_{s} v_{s1} \\
\end{pmatrix}
\]
We get $\det M_{\mathcal{Z}}=\sum_{i} e_i \mu_{i}=1$ by the choice of $v_{ij}= \frac{e_i\mu_i}{\mu_1 \cdots \mu_s}$.

By the inductive construction above we obtain $M_{\mathcal{Z}}$ from $M_{\mathcal{Z}}'$ by adding first a column of zeros and 
then replacing the row $(0\cdots 0, \mu_{jr_j-1}', 0\cdots 0, \mu_{jr_j-1}'v'_{jr_j-1},0)$ by the two rows
\[
\begin{matrix}
  0&\cdots &0& \mu_{jr_j-1}& 0\cdots 0 & {\mu_{jr_j-1}}v'_{jr_j-1}&-\beta\\
  0&\cdots &0& \mu_{jr_j}& 0\cdots 0 & {\mu_{jr_j}}v'_{jr_j-1}& \alpha\\
\end{matrix}
\]
Via row operations
 these rows transform to
\[
\begin{matrix}
  0&\cdots &0&0& 0\cdots 0 & 0 &1\\
  0&\cdots &0& \mu'_{jr_j-1}& 0\cdots 0 & \mu_{jr_j-1}'v'_{jr_j-1}& 0 
\end{matrix}
\]

Hence, we have $\det M_{\mathcal{Z}} = \det  M'_{\mathcal{Z}}$. But   $\det  M'_{\mathcal{Z}}= 1$ holds by induction.
\end{proof}

For the case $Y=\PP^1$ we obtain a complete classification. Now,
$z_1,\ldots, z_s$ are points in $\PP^1$. Without loss of generality,
we may assume that the support of $\D$ consists of at least $3$
points---in the other case $X$ would be toric. But this would imply
that $X = \A^n$.
 
By applying an isomorphism of $\PP^1$ we may assume, that
$z_1=\infty$, $z_2=0$ and $z_3,\ldots,z_s \in k^*$.  
Via $K(\PP^1)\cong k(t)$ we get $\divisor(t)= [0]-[\infty] =z_2-z_1$.

\begin{corollary}[Thm. 1.9., \cite{HHS09}]
\label{sec:cor-ufd-cplx-1}
  Every normal k-algebra $A$ of dimension $n$ admitting a (positive) grading by $\NN^{n-1}$ such that $A_0=k$ is factorial iff it is isomorphic to a free algebra over some
\[
A[\mathcal{Z}]=k\big[T_{ij} \;\big|\; 0 \leq i \leq s;\; 1 \leq j \leq r_{i}\big]/\big(T_i^{\mu_i} + T_2^{\mu_2} - z_iT_1^{\mu_1} \;\big|\;3 \leq i \leq s \big).
\]
In particular every such $k$-algebra it is a complete intersection of dimension $2 + \sum_i (r_i-1)$.
\end{corollary}
\begin{proof}
We consider $X:= \spec A$, which gives rise to a polyhedral divisor $\D$ on $\PP^1$ with support $z_1, \ldots, z_s \in \PP^1$. We consider the prime divisors $D_{i1},\ldots, D_{ir_i} \subset X$ corresponding to the vertices in $v_{i1}, \ldots v_{ir_i} \in \D_{z_i}^{(0)}$ and the prime divisors $E_1, \dots, E_r$ corresponding  to the extremal rays of $\D$. Moreover, let $\mu_{ij}$ denote the multiplicity $\mu(v_{ij})$.

Since $A$ is factorial, there are homogeneous elements $T_{ij}, S_i \in A$ such that $\divisor(T_{ij})=D_{ij}$ and $\divisor(S_{i})=E_{i}$. Moreover, these elements are unique up to multiplication by a constant.

First we show, that $A$ is generated by these elements. Let $f \cdot \chi^u$ be a homogeneous element of $A \cong A[\D]$. Then $\divisor(f \cdot \chi^u)$ is effective  and we have a decomposition
 \[\divisor(f \cdot \chi^u) = \sum_{ij} \alpha_{ij} D_{ij} + \sum_i \beta_i E_i + 
 c \cdot \pi^*z_1 +\sum_{z \notin \supp \D} c_z \cdot (D_{z,0}-\pi^*z_1),\]
with $c=\sum_z c_z$.  We find a polynomial $g_c$ of degree $C$ such that $\divisor(g)=\sum_z c_z \cdot (z-z_1)$. We get $\divisor((z_i-t) \cdot \chi^0) = \pi^*z_i - \pi^*z_1 = \sum_{j=1}^{r_i}\mu_{ij} D_{ij}-\sum_{j=1}^{r_1}\mu_{1j} D_{1j}$. Hence, we have $\divisor((z_i-t) \cdot \chi^0) = \divisor(\nicefrac{T_i^{\mu_i}}{T_1^{\mu_2}})$ and moreover
\[
  \divisor(f \cdot \chi^u) = \divisor\left({\textstyle \prod_{ij}T_{ij}^{\alpha_{ij}} \cdot \prod_i S_i^{\beta_i}  \cdot (T_1^{\mu_1})^c \cdot g_c\left(\nicefrac{T_2^{\mu_2}}{T_1^{\mu_1}}\right)}\right)\,.
\]
Since $A$ is factorial and the only invertible functions are constants
we get a equality not only for the induced divisors but also for the
functions themselves---at least up to a constant factor. Hence, $A$ is
generated by $T_{ij}$ and $S_i$.

Next, we will show that the relations between these elements are as
given above. Remember, that $T_{ij}$ was given only up to a constant
factor. By making a suitable choice of constants, we may assume that
$\nicefrac{T_i^{\mu_i}}{T_1^{\mu_1}} = (z_i-t) \chi^0$ holds for $2
\leq i \leq s$. In particular, the $\ZZ^{n-1}$ degree of $T_i^{\mu_i}$
does not depend on $i$. Now, the relation $(z_i-t) + t = z_i$ lifts to
$(z_i-t)\cdot \chi^0 + t\cdot \chi^0 = z_i\cdot \chi^0$ and we obtain
\[{T_i^{\mu_i}}/{T_1^{\mu_1}} + {T_2^{\mu_2}}/{T_1^{\mu_1}} = z_i\cdot
\chi^0 \quad \Rightarrow \quad T_i^{\mu_i} + T_2^{\mu_2} - z_i
T_1^{\mu_1}=0\,.
\]
It remains to show, that there are no other independent
relations. Therefore it is enough to consider a homogeneous relation $F(T,S)=  \sum_i c_i T^{\alpha_i} S^{\beta_i} = 0$. 
Being homogeneous implies, that
 $T^{\alpha_i}S^{\beta_i}/T^{\alpha_j}S^{\beta_j} = f \cdot
 \chi^0$. Hence, by our for principal divisors on page~\pageref{eq:principal-divisor} we get $\divisor(T^{\alpha_i}S^{\beta_i}/T^{\alpha_j}S^{\beta_j}) = \sum_{\ell} (\ord_{z_\ell} f)  \cdot \sum_{m} \mu_{\ell m} D_{\ell
   m}$. Thus, \[
 T^{\alpha_i}S^{\beta_i}/T^{\alpha_j}S^{\beta_j} =
 \prod_{\ell} \left(T_\ell^{\mu_\ell}\right)^{(\ord_{z_\ell} f)} \qquad
and
\qquad
F(T,S) = T^{\alpha'} S^{\beta'} \cdot F'(T_1^{\mu_1}, \ldots, T_r^{\mu_r}).\]
First we may divide by $T^{\alpha'} S^{\beta'}$. Now, by reducing with relations of the form $T_i^{\mu_i} + T_2^{\mu_2} - z_iT_1^{\mu_1}$ we eliminate the occurrences of $T_i$ with $i>2$ and come to homogeneous relation $F''(T_1^{\mu_1}, T_2^{\mu_2})=0$ of degree $m$. Dividing by $T_1^m$ gives
\[
F''(1\cdot \chi^0,T_2^{\mu_2}/T_1^{\mu_1})=F''(1,z)\cdot \chi^0=0.
\]
But this implies that $F'' = 0$ hold and hence that the relation $F$
is a combination of the relations $T_i^{\mu_i} + T_2^{\mu_2} -
z_iT_1^{\mu_1}$.
\end{proof}

\begin{remark}
  We can easily identify the log-terminal singularities of the form
  $A(\mathcal{Z})$---by theorem~\ref{sec:thm-log-terminal}
  $A(\mathcal{Z})$ is log-terminal if and only if $\max \mu_i > 1$ for
  at most three $1 \leq i_1 < i_2 < i_3 \leq s$ and $(\max \mu_{i_1},
  \max \mu_{i_2}, \max \mu_{i_3})$ is one of the platonic triples
  $(1,p,q)$, $(2,2,q)$, $(2,3,3)$, $(2,3,4)$, $(2,3,5)$.
\end{remark}

In the case of complexity one we are also able to characterize \emph{isolated} factorial singularities.  Every (normal) factorial surface singularity is of course isolated. For the remaining cases we provide the following theorem




\begin{theorem}
  Every factorial T-variety of complexity one and dimension at least
  three having an isolated singularity at the vertex is one of one of
  the following
\begin{enumerate}
\item A $\mathbf{cA_q}$ threefold singularities of the form
\[k[T_1,\ldots, T_4]/(T_1T_2 + T_3^{q+1} + T_4^r)\]
with $0 < q < r$ being coprime.
\item The fourfold singularity which is stably equivalent to $\mathbf{A_q}$:
\[
k[T_1,\ldots, T_5]/(T_1T_2 + T_3T_4 + T_5^{q+1})
\]
\item The fivefold singularity which is stably equivalent to $\mathbf{A_1}$:
\[k[T_1,\ldots, T_6]/(T_1T_2 + T_3T_4 + T_5T_6)
\]
\end{enumerate}
\end{theorem}
\begin{proof}
  We consider the Jacobian of a algebra as in corollary~\ref{sec:cor-ufd-cplx-1}.
  \[
  \left( 
    \begin{array}{cccccccccccccccc}
    z_3 f_{11} &  \cdots & z_3f_{1r_1} & f_{21} &  \cdots & f_{2r_2} &f_{31} &  \cdots & f_{3r_3} & & & & & \\
    z_4 f_{11} &  \cdots & z_4f_{1r_1} & f_{21} &  \cdots & f_{2r_2} & &   &  & f_{41} &  \cdots & f_{4r_4}& &  \\
\vdots &  & \vdots & \vdots  &   & \vdots & & & & & & & \ddots & \\
 z_s f_{11} &  \cdots & z_r f_{1r_1} & f_{21} &  \cdots & f_{2r_2} &  &   &  & &  &  &  &  f_{s1} &  \cdots & f_{sr_s} 
  \end{array}\right)
  \]
Here, $f_{ij}$ denotes the partial derivative $\nicefrac{\partial T_i^{\mu_i}}{\partial T_{ij}}$. Since we consider singularities of dimension at least three we must have  $r_\ell>1$ for at least one $\ell$. Then for 
\[
T_{ij}=
\begin{cases}
  1, & \text{if } (i,j)=(\ell,1),\\
 0, &  \text{otherwise.}
\end{cases}
\]
all but one column vanish. Hence, we are in the case of a hypersurface, because the matrix has to have full rank.

Now, one easily checks that a multiexponent $\mu_i > (1,1)$ automatically leads to partial derivatives $f_{ij}$ which jointly vanish even if one of the $T_{ij}$ does not vanish. Hence, the singular locus has dimension at least one.

Alternatively one could use Theorem~\ref{sec:thm-isolated-sing} and
Proposition~\ref{sec:prop-D-smooth} to prove the claim.
\end{proof}

\bibliographystyle{alpha}
\bibliography{tsing.bib}

\end{document}